\documentclass[smallextended,final]{svjour3}
\usepackage[letterpaper,top=1.5in, bottom=1in, left=1in, right=1in]{geometry}

\usepackage{epsfig,epsf,fancybox}
\usepackage{amsmath}
\usepackage{mathrsfs}
\usepackage{amssymb}
\usepackage{amsfonts}
\usepackage{graphicx}
\usepackage{color}
\usepackage[linktocpage,pagebackref,colorlinks,linkcolor=blue,anchorcolor=blue,citecolor=blue,urlcolor=blue,hypertexnames=false]{hyperref}
\usepackage{boxedminipage}
\usepackage{stmaryrd}
\usepackage{multirow}
\usepackage{booktabs}
\usepackage{cite}
\usepackage{float}
\usepackage[ruled,vlined,linesnumbered]{algorithm2e}

\usepackage[shortlabels]{enumitem}

\usepackage{tikz}
\usetikzlibrary{fit, arrows, shapes, positioning, shadows, matrix, calc}

\tikzstyle fwd=[line width=0.7pt, ->]   
\tikzstyle fwddash=[line width=0.7pt, dashed, ->]   
\tikzstyle bwd=[double, line width=0.3pt, ->]  
\tikzstyle refl=[double,  dashed, line width=0.2pt, ->]    

\tikzstyle{every node}=[font=\small] 

\tikzset{
  >=stealth', 
  invisible/.style={opacity=0}, 
  alt/.code args={<#1>#2#3}{\alt<#1>{\pgfkeysalso{#2}}{\pgfkeysalso{#3}}}, 
  visible on/.style={alt=#1{}{invisible}}, 
  smallnode/.style={circle, fill=black, thick, inner sep=1pt, minimum size=1.5pt}, 
  punkt/.style={
           rectangle,
           rounded corners,
           draw=black, very thick,
           text width=5em,
           minimum height=2em,
           text centered},
}

\tikzset{%
        brace/.style = { decorate, decoration={brace, amplitude=5pt} },
       mbrace/.style = { decorate, decoration={brace, amplitude=5pt, mirror} },
        label/.style = { black, midway, scale=0.5, align=center },
     toplabel/.style = { label, above=.5em, anchor=south },
    leftlabel/.style = { label,rotate=-90,left=.5em,anchor=north },   
  bottomlabel/.style = { label, below=.5em, anchor=north },
        force/.style = { rotate=-90,scale=0.4 },
        round/.style = { rounded corners=2mm },
       legend/.style = { right,scale=0.4 },
        nosep/.style = { inner sep=0pt },
   generation/.style = { anchor=base }
}

\usetikzlibrary{positioning}
\usetikzlibrary{shadows}
\usetikzlibrary{calc,positioning,shadows.blur,decorations.pathreplacing}
\usepackage{etoolbox}

\usepackage{mathtools}

\usepackage[normalem]{ulem} 

\newcommand{\bm}[1]{\boldsymbol{#1}}

\newcommand{\vg}{{\mathbf{g}}}

\newcommand{\vm}{{\mathbf{m}}}

\newcommand{\vv}{{\mathbf{v}}}

\newcommand{\vx}{{\mathbf{x}}}
\newcommand{\vy}{{\mathbf{y}}}
\newcommand{\vz}{{\mathbf{z}}}

\newcommand{\vG}{{\mathbf{G}}}

\newcommand{\vI}{{\mathbf{I}}}

\newcommand{\vV}{{\mathbf{V}}}

\newcommand{\cH}{{\mathcal{H}}}
\newcommand{\cI}{{\mathcal{I}}}

\newcommand{\cN}{{\mathcal{N}}}

\newcommand{\EE}{\mathbb{E}} 
\newcommand{\RR}{\mathbb{R}} 
\newcommand{\vzero}{\mathbf{0}} 
\newcommand{\vone}{{\mathbf{1}}} 

\newcommand{\Prob}{{\mathrm{Prob}}} 
\newcommand{\Diag}{{\mathrm{Diag}}} 
\newcommand{\Proj}{{\mathrm{Proj}}} 

\DeclareMathOperator*{\argmin}{arg\,min} 
\DeclareMathOperator*{\Argmin}{Arg\,min} 

\DeclareMathOperator*{\Min}{minimize}

\newcommand{\bc}{\begin{center}}
\newcommand{\ec}{\end{center}}

\newcommand{\bdm}{\begin{displaymath}}
\newcommand{\edm}{\end{displaymath}}

\newcommand{\beq}{\begin{equation}}
\newcommand{\eeq}{\end{equation}}

\newcommand{\bfl}{\begin{flushleft}}
\newcommand{\efl}{\end{flushleft}}

\newcommand{\bt}{\begin{tabbing}}
\newcommand{\et}{\end{tabbing}}

\newcommand{\beqn}{\begin{eqnarray}}
\newcommand{\eeqn}{\end{eqnarray}}

\newcommand{\beqs}{\begin{align*}} 
\newcommand{\eeqs}{\end{align*}}  

\newtheorem{assumption}{Assumption}

\numberwithin{equation}{section}

\usepackage{subcaption}

\usepackage{fontawesome}

\begin{document}

\title{Parallel and distributed asynchronous adaptive stochastic gradient methods}

\author{Yangyang Xu \and Yibo Xu \and Yonggui Yan \\  Colin Sutcher-Shepard \and Leopold Grinberg \and Jie Chen}

\institute{Yangyang Xu,  Y. Yan, C. Sutcher-Shepard. Department of Mathematical Sciences, Rensselaer Polytechnic Institute\\
Yibo Xu. School of Mathematical and Statistical Sciences, Clemson University. Part of Yibo's work was done when he was a postdoctoral fellow at Rensselaer Polytechnic Institute.\\
L. Grinberg. AMD, Cambridge, Massachusettes\\
J. Chen. MIT-IBM Watson AI Lab, IBM Research\\
correspondence to Yangyang Xu at \email{xuy21@rpi.edu}}

\date{\today}

\maketitle

\begin{abstract}
Stochastic gradient methods (SGMs) are the predominant approaches to train deep learning models. The adaptive versions (e.g., Adam and AMSGrad) have been extensively used in practice, partly because they achieve faster convergence than the non-adaptive versions while incurring little overhead. On the other hand, asynchronous (async) parallel computing has exhibited significantly higher speed-up over its synchronous (sync) counterpart. Async-parallel non-adaptive SGMs have been well studied in the literature from the perspectives of both theory and practical performance. Adaptive SGMs can also be implemented without much difficulty in an async-parallel way. However, to the best of our knowledge, no theoretical result of async-parallel adaptive SGMs has been established. The difficulty for analyzing adaptive SGMs with async updates originates from the second moment term. In this paper, we propose an async-parallel adaptive SGM based on AMSGrad. We show that the proposed method inherits the convergence guarantee of AMSGrad for both convex and non-convex problems, if the staleness (also called delay) caused by asynchrony is bounded. Our convergence rate results indicate a nearly linear parallelization speed-up if $\tau=o(K^{\frac{1}{4}})$, where $\tau$ is the staleness and $K$ is the number of iterations. The proposed method is tested on both convex and non-convex machine learning problems, and the numerical results demonstrate its clear advantages over the sync counterpart and the async-parallel nonadaptive SGM.

\vspace{0.3cm}

\vspace{0.3cm}


\end{abstract}

\section{Introduction}
In recent years, \emph{adaptive} stochastic gradient methods (SGMs), such as AdaGrad \cite{duchi2011adaptive}, Adam \cite{kingma2014adam}, and AMSGrad \cite{reddi2019convergence}, have become very popular due to their great success in training deep learning models. These adaptive SGMs can practically be significantly faster than a classic \emph{non-adaptive} SGM. We aim at speeding up adaptive SGMs on massively parallel computing resources. One way is to parallelize them in a \emph{synchronous} (sync) way by using a large batch size, in order to obtain high parallelization speed-up. However, it has been observed \cite{keskar2016large, masters2018revisiting} that large-batch training in deep learning can often lead to worse generalization than small-batch training. To simultaneously gain fast convergence, high parallelization speed-up, and also good generalization, we propose to develop \emph{asynchronous} (async) parallel \emph{adaptive} SGMs.

Async-parallel computing under either shared-memory or distributed setting has been demonstrated to enjoy significantly higher speed-up than its sync counterpart, e.g., \cite{recht2011hogwild, lian2015asynchronous, liu2014asynchronous-cd, peng2016arock}. 
At each iteration of a sync-parallel method, the workers that finish tasks earlier must wait for those that finish later. This can result in a lot of idle waiting time. In addition, under a shared-memory setting, all workers access the memory simultaneously, which can cause memory congestion \cite{bertsekas1991some}, and under a distributed setting, enforcing synchronization is often inefficient
due to communication latency. For these reasons, a sync-parallel method may have a very low parallelization speed-up. On the contrary, an async-parallel method does not require all workers to keep the same pace and can eliminate the waiting time and the memory congestion issue. However, it may be difficult to guarantee the convergence of an async-parallel method, 
because outdated information could be used in updating the variables.

Async-parallel methods have been developed for non-adaptive SGMs, e.g., in \cite{agarwal2011distributed, recht2011hogwild, lian2015asynchronous}. However, a \emph{parallel nonadaptive} SGM may be slower than a \emph{non-parallel adaptive} SGM to reach the same accuracy. Hence, it is important to design a method that can achieve the high speed-up of async-parallel implementation and also the fast convergence of an adaptive SGM. How to guarantee a successful integration remains an open question, although numerical experiments have been conducted to demonstrate the performance of async-parallel adaptive SGMs, e.g., in \cite{dean2012large, guan2017delay}. The non-triviality lies in the integrated analysis of the second-moment term used in adaptive SGMs. In this work, we give an affirmative answer to the question, by designing an async-parallel adaptive SGM under both shared-memory and distributed settings. 


\subsection{Proposed algorithm}
We consider the stochastic program
\begin{equation}\label{eq:stoc-prob}
F^* = \Min_{\vx\in X} ~F(\vx):= \EE_\xi \big[f(\vx;\xi)\big],
\end{equation}
where $\xi\in \Xi$ is a random variable, and $X\subseteq \RR^n$ is a closed convex set. When $\xi$ is uniformly distributed on a finite set $\Xi=\{\xi_1,\ldots, \xi_N\}$, \eqref{eq:stoc-prob} reduces to a finite-sum structured problem, which includes as examples all machine learning problems with pre-collected training data.

For solving \eqref{eq:stoc-prob}, we propose an async-parallel adaptive SGM, named APAM, which is based on AMSGrad in \cite{reddi2019convergence}. We adopt a master-worker set-up. The pseudocode is shown in Algorithm~\ref{alg:async-adp-sgm}, which is from the master's view. The updates in \eqref{eq:update-m} through \eqref{eq:update-x} are performed by the master, while the workers compute the stochastic gradients $\{\vg^{(k)}\}$. Due to the potential information delay caused by asynchrony, $\vg^{(k)}$ may not be evaluated at $\vx^{(k)}$; see more discussions in section~\ref{sec:staleness}. In \eqref{eq:update-x}, we define a weighted norm as $\|\vx\|_\vv^2 :=\vx^\top \Diag(\vv) \vx$, and if $X=\RR^n$, the update reduces to $\vx^{(k+1)}=\vx^{(k)} - \alpha_k \vm^{(k)}\oslash \sqrt{\widehat\vv^{(k)}}$, where $\oslash$ denotes component-wise division. The weight vector $\widehat\vv^{(k)}$ depends on all previous stochastic gradients, and thus the effective learning rate $\alpha_k \vone \oslash \sqrt{\widehat\vv^{(k)}}$ \emph{adaptively} depends on the gradients. 


\begin{algorithm}
\caption{\textbf{a}sync-\textbf{p}arallel \textbf{a}daptive stochastic gradient \textbf{m}ethod (APAM) from master's view}
\DontPrintSemicolon
\label{alg:async-adp-sgm}
\textbf{Initialization:} choose $\vx^{(1)}\in X$ and $\beta_1, \beta_2\in [0,1)$; set $\vm^{(0)}=\vzero$ and $\vv^{(0)}=\widehat\vv^{(0)}=\vzero$, .\;
\For{$k=1,2,\ldots$}{
Obtain a (possibly outdated) stochastic gradient $\vg^{(k)}$ from a worker, and update
\begin{align}
&\vm^{(k)}=\beta_1 \vm^{(k-1)} + (1-\beta_1)\vg^{(k)},\label{eq:update-m}\\
&\vv^{(k)}=\beta_2 \vv^{(k-1)} + (1-\beta_2)\big(\vg^{(k)}\big)^2,\label{eq:update-v}\\
&\widehat\vv^{(k)}=\max\big\{\widehat\vv^{(k-1)},\, \vv^{(k)}\big\}, \label{eq:update-vhat}\\
\label{eq:update-x}
&\vx^{(k+1)}\in\Argmin_{\vx\in X}\,\langle \vm^{(k)}, \vx\rangle + \frac{1}{2\alpha_k}\|\vx-\vx^{(k)}\|_{
	 \sqrt{\widehat\vv^{(k)}} 
}^2.
\end{align}
\vspace{-0.3cm}
}
\end{algorithm}

We emphasize the importance of the proposed method in training very large-scale deep learning models. It is well-known that adaptive SGMs converge significantly faster than a non-adaptive SGM; see \cite{duchi2011adaptive, kingma2014adam} for example or our numerical results in section~\ref{sec:numerical}. In addition, an async-parallel method can achieve much higher speed-up than its sync counterpart. Hence, it is paramount to design a method that can inherit advantages from both adaptiveness and async-parallelization, in order to efficiently train a very ``big'' deep learning model on multi-core or distributed-memory machines. 

We make the exploration on async-parallel adaptive SGM based on AMSGrad, because of its simplicity and nice numerical performance. Besides AMSGrad, there are several other adaptive SGMs in the literature, such as AdaGrad \cite{duchi2011adaptive}, RMSProp \cite{RMSprop2012}, Adam \cite{kingma2014adam}, Padam \cite{zhou2018convergence}, and AdaFom \cite{chen2018convergence}. While AdaGrad can have guaranteed sublinear convergence, its numerical performance can be significantly worse than AMSGrad, because the former simply uses $\vg^{(k)}$ instead of the exponential averaging gradient $\vm^{(k)}$ and also its effective learning rate can decay very fast. Adam can perform similarly or slightly better than AMSGrad, but its convergence is not guaranteed even for convex problems, due to a possibly too large learning rate. Padam is a generalized version of AMSGrad, and the performance of AdaFom is somehow between AdaGrad and AMSGrad. We believe that the convergence of AdaGrad, Padam and AdaFom can be inherited by their async versions.

\subsection{Related works}
In the literature, there are many works on SGMs. We briefly review those on async-parallel SGMs and adaptive SGMs, which are closely related to our work.

\vspace{0.1cm}

\noindent\textbf{Async-parallel non-adaptive SGM.}~~~The stochastic approximation method can date back to 1950's \cite{robbins1951stochastic} for solving a root-finding problem. The SGM, as a first-order stochastic approximation method, has been analyzed for both convex and non-convex problems; see \cite{nemirovski2009robust, polyak1992acceleration, ghadimi2013stochastic} for example. In order to achieve high speed-up, async-parallel SGM and/or distributed SGM with delayed gradient have been developed to solve problems that involve huge amount of data,  
e.g., in \cite{agarwal2011distributed,recht2011hogwild,lian2015asynchronous,feyzmahdavian2016asynchronous,mania2017perturbed, leblond2018improved}. 
The work \cite{agarwal2011distributed} assumes  
a distributed setting with a central node and analyzes the SGM with delayed stochastic gradients.  
\cite{lian2015asynchronous} studies the async-parallel SGM for non-convex optimization under both shared-memory and distributed settings. After obtaining a sample gradient, the shared-memory async-parallel method in \cite{lian2015asynchronous} needs to perform randomized coordinate update to avoid overwriting, because all threads are allowed to update the variables without coordination to each other. \cite{recht2011hogwild} also studies shared-memory async-parallel SGM. It does not require randomized coordinate update. However, its analysis relies on strong convexity of the objective and the assumption that the data involved in every sample function is sparse. \cite{leblond2018improved} further removes the sparsity requirement by providing an improved analysis for async-parallel stochastic incremental methods. In \cite{leblond2018improved}, a novel ``after read'' approach is introduced to order the iterate and address one independence issue between the random sample and the iterate that is read. 
\cite{backstrom2019mindthestep, sra2016adadelay} adapt the stepsize of the async SGM to the staleness of stochastic gradient, 
and \cite{lian2018asynchronous,wu2018error} explore the async SGM under a decentralized setting. 

\vspace{0.1cm}

\noindent\textbf{Adaptive SGM.}~~~
Adam \cite{kingma2014adam} is probably the most popular adaptive SGM. It was proposed for convex problems. However, the convergence of Adam is not guaranteed. To address the convergence issue, \cite{reddi2019convergence} makes a modification to the second-moment term in Adam and  proposes AMSGrad. It performs almost the same updates as those in \eqref{eq:update-m} through \eqref{eq:update-x}, with the only difference that AMSGrad uses non-fixed weights in computing $\vm^{(k)}$, i.e., it lets $\vm^{(k)}=\beta_{1,k}\vm^{(k-1)}+(1-\beta_{1,k})\vg^{(k)}$ for all $k\ge1$.  In order to guarantee sublinear convergence, \cite{reddi2019convergence} requires a diminishing sequence $\{\beta_{1,k}\}$, and to have a rate of $O(1/\sqrt{k})$, $\{\beta_{1,k}\}$ needs to decay as fast as $1/k$. However, \cite{reddi2019convergence} sets $\beta_{1,k} = \beta_1\in (0,1),\,\forall\, k\ge1$ in all its numerical experiments, and it turned out that the algorithm with a constant weight $\beta_1$ could perform significantly better than that with decaying weights. By new analysis, we will show, as a byproduct, that an $O(1/\sqrt{k})$ convergence rate can be achieved even with a constant weight. 
Later, \cite{tran2019convergence} proposes AdamX, which is similar to AMSGrad but addresses a flaw in the analysis of AMSGrad.  
AdamX embeds $\beta_{1,k}$ in updating $\widehat{\vv}^{(k)}$. However, it still requires a decaying $\beta_{1,k}$ to guarantee sublinear convergence. 
To have nice generalization, \cite{chen2018closing} proposes Padam that includes AMSGrad as a special case. It uses $-\vm^{(k)}\oslash(\widehat\vv^{(k)})^p$ as the search direction, where $p\in(0,0.5]$. When $p=\frac{1}{2}$, Padam reduces to AMSGrad. It was demonstrated that $p=\frac{1}{8}$ could yield the best numerical performance. 
To avoid extremely large or small learning rates, \cite{luo2019adaptive} proposes variants of Adam and AMSGrad by keeping the second-moment term in nonincreasing intervals. Asymptotically, they approach to non-adaptive SGMs. 
For strongly-convex online optimization, \cite{fang2019convergence}  
presents a variant of AMSGrad, and \cite{wang2020sadam} proposes SAdam, as a variant of Adam. 
For non-convex problems, \cite{chen2018convergence} gives a general framework of Adam-type SGMs and establishes convergence rate results. 
Padam is extended in \cite{zhou2018convergence} to non-convex cases. 
%
%
\cite{nazari2019dadam} presents a variant of AMSGrad by introducing one more moving-average term in the update of $\widehat \vv$, and 
the analysis is conducted for both convex and non-convex problems. 

\vspace{0.1cm}

\noindent\textbf{Async-parallel adaptive SGM.}~~~The async-parallel implementation of AdaGrad is explored in \cite{dean2012large}. Experimental results on training deep neural networks are shown to demonstrate the performance of the async-parallel AdaGrad. However, no convergence analysis is given in \cite{dean2012large}, and in addition, AdaGrad often performs significantly worse than AMSGrad. \cite{guan2017delay} proposes a delay-compensated asynchronous Adam, which exhibits advantages over an asynchronous nonadaptive SGM for solving deep learning problems. However, the theoretical result in \cite{guan2017delay} does not guarantee convergence to stationarity but simply implies that the expected value of gradient norm can be bounded.

For the readers' convenience, we compare, in Table~\ref{table:related}, APAM to several closely related methods based on a few important ingredients about the algorithms and the targeted problem. 

\begin{table}[h]
	\caption{ {\small A comparison of ingredients among several algorithms for solving problems in the form of \eqref{eq:stoc-prob}. In the second column, ``$F$ \& Constraint $X$'' reflects the underlying assumption on $F$ and feasibility constraint $X$:  ``cvx'' for convexity, ``noncvx'' for non-convexity, ``yes'' for closed convex constraint $X$, and ``no'' for unconstrained problems.
	In the third column, ``Adaptivity'' reflects whether the algorithm implements adaptivity. In the fourth column, ``Weights'' reflects the restriction on the momentum parameters in the adaptive algorithms: ``constant'' indicates a constant parameter choice (i.e., $(\beta_{1,k},\,\beta_{2,k})=(\beta_1,\beta_2), \forall\, k$), and ``decreasing'' indicates a decreasing parameter choice.  
	In the fifth column, ``Async. delayed'' reflects whether the algorithm has a convergence guarantee for its asynchronous implementation with delayed gradient information. In the last column, convergence rate results for both convex and non-convex models are listed: $\tau$ for the upper bound on the delay and $K$ for the total number of iterations; for convex models, the convergence is measured by the expected objective gap, while for non-convex models, it is measured by the expected stationarity violation. 
	Specifically, AdaDelay \cite{sra2016adadelay} has the assumption that the delay has 
	a bounded expectation $\EE[\tau_k]=\bar\tau < \infty$ and a bounded second moment $\EE[\tau_k^2]=\Omega(\bar\tau^2)$. 
		} 
	} 
	
	\label{table:related}
	\begin{center}
		\resizebox{.98\textwidth}{!}{
			\begin{tabular}{|c|c|c|c|c|c|}
				\hline
				Method                 & $F$ \& Constraint $X$ & Adaptivity & Weights $(\beta_{1,k},\,\beta_{2,k})$ & Async. delayed &  Order of convergence rate  \\ \hline\hline
				Mirror descent \cite{agarwal2011distributed}               & cvx \& yes       & no                 &  ---               &     yes    &  ---  \\
				AMSGrad \cite{reddi2019convergence} & cvx \& yes & yes & $\beta_{1,k}=\beta_1/k,\,\beta_1<\sqrt{\beta_2}$ & no & --- \\
				AdamX \cite{tran2019convergence} & cvx \& yes        & yes                  &      $\beta_{1,k}=\beta_1/k,\,\beta_1<\sqrt{\beta_2}$          &    no    &  ---  \\
				Padam \cite{chen2018closing} & cvx \& no        & yes                  &      $\beta_1< \beta_2^{2p},\, p\in[0,1/2]$          &    no    &  ---  \\
				AdaDelay \cite{sra2016adadelay} & cvx \& yes & no & --- & yes & $(\sqrt{1+\bar\tau} +\bar\tau^4/\sqrt{K})/{\sqrt{K}}$ \\
				AsySG-con \cite{lian2015asynchronous}       & noncvx \& no       & no                  &       ---     & yes   & $(1+\tau /\sqrt{K})/\sqrt{K}$  \\
				AMSGrad \& AdaFom \cite{chen2018convergence}      & noncvx \& no      & yes                  &       constant or decreasing     & no   & ---  \\\hline
				\multirow{ 2}{*}{APAM (this paper)} 
				& cvx \& yes       & yes                 &       constant            &    yes   & $(1+\tau^2 /\sqrt{K})/\sqrt{K}$  \\
				& noncvx \& no       & yes                 &       constant          &    yes   & $(1+\tau /K^{1/4}+\tau^2 /\sqrt{K})/\sqrt{K}$  \\[0.1cm]\hline       
			\end{tabular}
		}
	\end{center}
\end{table}

\subsection{Contributions}
Our contributions are three-fold. \emph{First}, we propose an async-parallel adaptive SGM, named APAM, which is an asynchronous version of AMSGrad in \cite{reddi2019convergence}. APAM works under both shared-memory and distributed settings. For both settings, we adopt a master-worker architecture. Only the master updates model parameters, while the workers compute stochastic gradients asynchronously. APAM is lock-free. The master can perform updates while the workers are reading/receiving variables, and also since only the master updates variables, there is no need to lock the writing process. 
To the best of our knowledge, APAM is the first async-parallel adaptive SGM that maintains the fast convergence of an adaptive SGM and also achieves a high parallelization speed-up.
\emph{Secondly}, we analyze the convergence rate of APAM for both convex and non-convex problems. For convex problems, we establish a sublinear convergence result in terms of the objective error, and for non-convex problems, we show a sublinear convergence result in terms of the violation of stationarity. The established results indicate that the staleness $\tau$ has little impact on the convergence speed if it is dominated by $K^{\frac{1}{4}}$, where $K$ is the maximum number of iterations. Therefore, if $\tau=o(K^{\frac{1}{4}})$, a nearly-linear speed-up can be achieved, and this is demonstrated by numerical experiments.
\emph{Thirdly}, over the course of analyzing APAM, we also conduct new convergence analysis for AMSGrad. Our convergence rate results do not require a diminishing sequence to weigh the gradients. In practice, constant weights are almost always adopted. Hence, our results bring the theory closer to practice.

\subsection{Notation and outline} We use lower-case bold letter $\vx,\vy,\ldots$ for vectors. The $i$-th component of a vector $\vx$ is denoted as $x_i$. 
For any two vectors $\vx$ and $\vy$ of the same size, ${\vx}\odot{\vy}$ denotes a vector by component-wise multiplication, and ${\vx}\oslash{\vy}$ denotes a vector by the component-wise division, 
with $\frac{0}{0}=0$. For any $\vv\ge\vzero$, 
$\sqrt{\vv}$ or $(\vv)^{\frac{1}{2}}$ denotes a vector by the component-wise square root. We add a superscript $^{(k)}$ to specify the iterate, i.e., $\vx^{(k)}$ denotes the $k$-th iterate. $\Diag(\vv)$ denotes the diagonal matrix with $\vv$ as the diagonal vector. 
Given $\vv\ge\vzero$, $\|\vx\|_\vv^2 :=\vx^\top \Diag(\vv) \vx$, and  
$\Proj_{X,\vv}(\vx) := \argmin_{\vy\in X}\|\vy-\vx\|_\vv^2.$ 
We use $\|\cdot\|$ for the Euclidean norm of a vector and also the spectral norm of a matrix. 
$[n]$ denotes $\{1,\ldots,n\}$, and for a subset $A\subseteq[n],$ $A^{c}$ denotes the complement set of $A$. $\tilde \nabla f(\vx)$ denotes a subgradient of $f$ at $\vx$, and it reduces to the gradient $\nabla f(\vx)$ if $f$ is differentiable. We let $\cH_k$ be the $\sigma$-algebra generated by $\{\vx^{(t)}\}_{t\le k}$.

\vspace{0.1cm}
\noindent\textbf{Outline.}~~The rest of the paper is outlined as follows. In section~\ref{sec:imp}, we give details on how to implement the proposed algorithm. Convergence analysis is given in section~\ref{sec:cvx} for convex problems and in section~\ref{sec:ncvx} for non-convex problems, and numerical results are shown in section~\ref{sec:numerical}. Finally, we conclude the paper in section~\ref{sec:conclusion}.

\section{Implementation of the proposed method}\label{sec:imp}

In this section, we give more details on how to implement Algorithm~\ref{alg:async-adp-sgm} and also how the delay happens as the workers run asynchronously in parallel.

\subsection{Organization of master and workers}
We first explain how the master and workers communicate under a shared-memory or distributed setting.  

\vspace{0.1cm}
\noindent\textbf{Shared-memory setting.}~~Suppose that there are multiple processors and all the data and variables (or model parameters) are stored in a global memory. We assign one or a few as the \emph{master}(\emph{s}). The updates to $\vx, \vm, \vv$ and $\widehat\vv$ in Algorithm~\ref{alg:async-adp-sgm} are all performed by the master(s), while the computation of $\vg$ is done by other processors (called \emph{worker}s). 
See the left of Figure \ref{fig:architecture} for an illustration.   
Every worker reads $\vx$ and data from the global memory, computes a stochastic gradient $\vg$, and saves it in a pre-assigned memory. If there is a $\vg$ that has not been used, then the master acquires it. Otherwise, the master computes one stochastic gradient by itself. We allow more than one processor to serve as the master in case one is not fast enough to digest the $\vg$ vectors fed by the workers. In the case of multiple master processors, we will partition the vectors into blocks and let one master processor update one block, and we synchronize all the master processors while performing the updates. However, we never synchronize the workers. 

Our shared-memory set-up is fundamentally different from existing ones, e.g., in \cite{lian2015asynchronous, recht2011hogwild, leblond2018improved}, which allow all processors to update the variables. Without coordination between the processors, overwriting issue will arise if all processors write to the memory at the same time. To avoid the issue, these existing works need to perform randomized coordinate updates \cite{lian2015asynchronous}, or require sparsity of the stochastic gradient \cite{recht2011hogwild}, or assume strong convexity of the objective \cite{leblond2018improved}. However, in training a deep learning model, neither the sparsity condition nor the strong convexity assumption will hold. In addition, the coordinate update will be inefficient because the whole $\vg$ is computed but just one or a few coordinate gradients are used. In contrast, our method does not have this issue due to the master-worker set-up. Furthermore, our set-up enables a simpler analysis without sacrificing the high parallelization speed-up.

\vspace{0.1cm}
\noindent\textbf{Distributed setting.}~~Suppose multiple processors do not share memory and hence data need to be transmitted through inter-process communication.
The master takes charge of updating $\vx, \vm, \vv$ and $\widehat\vv$. It sends $\vx$ to workers, and the workers compute and send stochastic gradients to the master for the update. See the right of Figure~\ref{fig:architecture} for an illustration. We assume that each worker has its own memory and can generate samples by the same distribution. 

\begin{figure}
\centering
\begin{subfigure}[b]{0.46\textwidth}
\centering
\begin{tikzpicture}[fill=blue!20, scale = 0.67]
\draw (-0.3,-2.5) [rounded corners, fill=red!10] rectangle (4.5, -1.5);
\draw (5.,-2.5) [rounded corners, fill=red!10] rectangle (10.8, -1.5);
\path (5,0) node (master) [rounded corners, rectangle, draw, fill=green!60] {master}
	(2.1,-2.) node {\footnotesize stochastic gradients $\vg$}
	(7.9,-2.) node {data and vectors $\vx,\vm,\vv,\hat\vv$}
	(1.5, -5.) node (w11) [rounded corners, rectangle, draw, fill] {worker}
	(5, -5.) node (w13)  {$\bm{\cdots\cdots}$}
	(8.5, -5.) node (w14) [rounded corners, rectangle, draw, fill] {worker};
	
\draw[ ->] (2,-1.4) -- (4.8,-0.5)
		node[pos = .5, rotate = 36, above]	 {};

\draw[ ->] (5.2,-0.5) -- (8,-1.4)  
		node[pos = .5, rotate = -35, above]	 {};	
		
\draw[ ->] 	(1.2, -4.5) -- (1.2, -2.6)		
		node[pos = .6, rotate = -90, above]	 {};
		
\draw[ ->] 	(4.6, -4.5) -- (1.7, -2.6)		
		node[pos = .55, rotate = -46, above]	 {};
		
\draw[ ->] 	(8.2, -4.5) -- (2, -2.6)		
		node[pos = .65, rotate = -25, above]	 {}; 
		
\draw[ ->] 	(7.5, -2.6)	 -- (1.6, -4.5)  	
		node[pos = .85, rotate = 27, above]	 {};
		
\draw[ ->] 	(8.3, -2.6)	 --  (5.0, -4.5)	
		node[pos = .82, rotate = 43, above]	 {};
		
\draw[ ->] 	(8.8, -2.6)	 --  (8.8, -4.5)		
		node[pos = .75, rotate = 90, above]	 {};						
\end{tikzpicture} 
\end{subfigure}\hfill
\begin{subfigure}[b]{0.49\textwidth}
\centering
\begin{tikzpicture}[fill=blue!20, scale = 0.67]
\path (0,0) node (master) [rounded corners, rectangle, draw, fill=green!60] {master}
	(-4, -3.5) node (agg1) [rounded corners, rectangle, draw, fill=red!25] {worker}
	(-1, -4.5) node (agg2) [rounded corners, rectangle, draw, fill=red!25] {worker}
	(2, -4.5) node (agg3)  {$\bm{\cdots\cdots}$}
	(4, -3.5) node (agg4) [rounded corners, rectangle, draw, fill=red!25] {worker};

\draw[ ->] (-4,-3.) -- (-1.3,-0.4)
		node[pos = .5, rotate = 50, above]	 {\footnotesize $\vg$};		
\draw[ ->] (-1,-4.1) -- (-0.5,-0.5)
		node[pos = .5, rotate = 83, above]	 {\footnotesize $\vg$};		
\draw[ ->] (2,-4.1) -- (0.5,-0.5);		
\draw[->] (4,-3.) -- (1.3,-0.4) 
		node[pos = .5, rotate = -50, above]	 {\footnotesize $\vg$};

\draw[ ->] (-1,-0.4) -- (-3.8,-3.1)
		node[pos = .5, rotate = 50, below]	 {\footnotesize $\vx$};		
\draw[ ->] (-0.2,-0.5) -- (-0.7,-4.1)
		node[pos = .5, rotate = 85, below]	 {\footnotesize $\vx$};		
\draw[ ->] (0.2,-0.5) -- (1.7,-4.1);		
\draw[->] (1,-0.4) --  (3.7,-3.1)
		node[pos = .5, rotate = -50, below]	 {\footnotesize $\vx$};

\end{tikzpicture} 
\end{subfigure}
\caption{{\small Shared memory setting (left) vs. distributed setting (right): a demonstration.}}\label{fig:architecture}
\end{figure}

\subsection{Iteration counter and staleness}\label{sec:staleness} Notice that Algorithm~\ref{alg:async-adp-sgm} is viewed from the perspective of the master. We use $k$ as the iteration counter. It increases by one whenever the master performs an update to $\vx$. Hence, $\vx^{(k)}$ denotes the iterate maintained by the master at the beginning of the $k$-th update, and $\vg^{(k)}$ is the stochastic gradient used in the $k$-th update. Since the master continuously updates $\vx$, after worker $\#i$ reads (or receives) the variable, the master may have already changed $\vx$ before it uses the stochastic gradient fed by worker $\#i$. Therefore, the stochastic gradient $\vg^{(k)}$ that is used to obtain $\vx^{(k+1)}$ may not be evaluated at the current iterate $\vx^{(k)}$ but at an outdated one. 
See Figure~\ref{fig:lockonoff} for an illustration.
More precisely, we have
\begin{equation}\label{eq:vg-k}
\textstyle \vg^{(k)}=\frac{1}{b_k}\sum_{i=1}^{b_k}\tilde\nabla f(\widehat \vx^{(k)}; \xi_i^{(k)}),
\end{equation}
where $b_k$ is the number of samples, and $\widehat \vx^{(k)}$ can be an outdated iterate or a mixture of several iterates; see \eqref{eq:shared} below for its expression. 

\begin{figure}[h]
\centering
\begin{subfigure}[t]{0.44\textwidth}
\centering
\begin{tikzpicture}[
	block/.style={
		draw,
		rectangle, 
		text width={width("update $\vx^{(2)}$")},text height={height("send $\vg^{(3)}$")},
		align=center,
		font=\small
	},
	fill=blue!20, scale = 0.4, every node/.style={scale=.7}]
	
\path 
(1.05,-2) node (x00) [ rectangle] {\Large $\ldots$}
(2.1,-2) node (x10) [ rectangle, draw, fill=blue!50, text height=60] { }
(6.5,-2) node (x11) [ rectangle, draw, fill=red!50, text height=60] { }
(9.05,-2.0) node (x12) [ rectangle, draw, fill=olive!50, text height=60] { }
(11.6,-2.0) node (x13) [ rectangle, draw, fill=violet!50, text height=60] { }
(3,-6.85) node[rectangle, draw, fill=blue!50, text width=30] (wr1l)  { }
(6.0,-6.85) node[draw,rectangle,fill=white] (wr1e) {compute $\vg^{(k+1)}$}
(12.3,-6.85) node[rectangle, draw, fill=violet!50] (wr1r)  {$\ldots$}
(2.1,-9.55) node[draw,rectangle,fill=white] (wr2e) {compute $\vg^{(k-1)}$}
(8.75,-9.55) node[rectangle, draw, fill=red!50, text width=60] (wr2r)  { }
(12.35,-9.55) node[draw,rectangle,fill=white] (wr2e2) {compute $\ldots$}
(3.7,-12.25) node[draw,rectangle,fill=white] (wr3e) {compute $\vg^{(k)}$}
(11.3,-12.25) node[rectangle, draw, fill=olive!50, text width=60] (wr3r)  { }
(2.1,.4) node(v1){{\color{blue}$\vx^{(k-1)}$}} 
(6.5,.4) node(v2){{\color{red}$\vx^{(k)}$}}
(9.05,.4) node(v3){{\color{olive}$\vx^{(k+1)}$}}
(11.6,.4) node(v4){{\color{violet}$\vx^{(k+2)}$}}

(-1.7,-6.85) node{\small Worker 1} 
(-1.7,-9.55) node{\small Worker 2} 
(-1.7,-12.25) node{\small Worker 3}
(-1.7,-2) node{\small Master};

\draw[dashed] (4.15,-0.0)--(4.15,-10.0);
\draw[dashed] (5.45,-0.0)--(5.45,-12.7);
\draw[dashed] (8.07,-0.0)--(8.07,-7.35);

\draw[dashed] (6.7,-0.0)--(6.7,-10);
\draw[dashed] (9.25,-0.0)--(9.25,-12.7);
\draw[dashed] (11.8,-0.0)--(11.8,-7.35);

\draw[-latex] (4.15,-2) -- (6.25,-2) ;
\draw[-latex] (6.7,-2) -- (8.8,-2) ;
\draw[-latex] (9.25,-2) -- (11.35,-2) ;

\end{tikzpicture} 
\caption{{\small \textbf{Delayed gradient in a distributed setting.} Each worker uses the $\vx$ vector it receives from the master to compute one $\vg$ vector. Due to asynchrony, the $\vg$ vector that master uses for update may not be calculated at the current $\vx$. For example, $\vg^{(k+1)}$ is computed by worker 1 at $\vx^{(k-1)}$ but used by master to update $\vx^{(k+1)}$ to $\vx^{(k+2)}$, which causes a delay of 2. 
}}\label{fig:distributed}
\end{subfigure}\hfill
\begin{subfigure}[t]{0.44\textwidth}
\centering
\begin{tikzpicture}[
block/.style={
draw,
rectangle, 
text width={width("calc. $\vx^{(1)}$")},text height={height("calc. $\vx^{(1)}$")},
align=center,
font=\small
},
fill=blue!20, scale = 0.4, every node/.style={scale=.7}]

	\path 
(2.5,-2) node (x11) [ rectangle] {\Large $\ldots$}
(4.0,-2) node (x12) [ rectangle, draw, fill=blue!50, text height=60] { }
(7.35,-2) node (x22) [ rectangle, draw, fill=red!50, text height=60] { }
	(2.2,-6.85) node[rectangle, draw, fill=white] (re1t)  {compute $\ldots$}
	(5.9,-6.85) node[rectangle, draw, fill=blue!50, text width=50] (re3t)  {}
	(8.1,-6.85) node[rectangle, draw, fill=red!50, text width=10] (re4t)  {}
	(2.9,-9.55) node[draw,rectangle,fill=white] (wr2e) {compute $\vg^{(k-1)}$}
	(3.9,-12.25) node[rectangle,draw, fill=white] (re1t)  {compute $\ldots$}
	(6.8,-12.25) node[rectangle, draw, fill=blue!50, text width=25] (wr3t)  { }
	(9.0,-12.25) node[rectangle, draw, fill=red!50, text width=35] (wr3e)  { }
(4.0,0.4) node(v1){{\color{blue}$\vx^{(k-1)}$}} 
(7.35,0.4) node(v2){{\color{red}$\vx^{(k)}$}}

(-2,-6.85) node{\small Worker 1} 
(-2,-9.55) node{\small Worker 2} 
(-2,-12.25) node{\small Worker 3}
(-2,-2) node{\small Master};

\draw[dashed] (5.0,0.1)--(5.0,-10);

\draw[-latex] (5.0,-2) -- (7.1,-2) ;
\end{tikzpicture} 
\caption{\small \textbf{Inconsistent reading in a shared memory setting.} Each worker reads $\vx$ from the shared memory and then computes one $\vg$ vector. Due to asynchrony and without lock, the read can be inconsistent. For example, while worker 1 reads $\vx^{(k-1)}$, the $\vx$ vector is updated by master to $\vx^{(k)}$ before worker 1 finishes its reading, and thus worker 1 reads a mixture of $\vx^{(k-1)}$ and $\vx^{(k)}$.}\label{fig:LOF2}
\end{subfigure}
\caption{{\small a demonstration of consistent but outdated read in the distributed setting (left subfigure) and inconsistent read in the shared memory setting (right subfigure).} 
		}\label{fig:lockonoff}
\end{figure}

\subsection{Consistent and inconsistent read}  
In the distributed setting, we have $\widehat \vx^{(k)}=\vx^{k-\tau_k}$ for some $\tau_k\ge0$ due to communication delay, i.e., the $\vx$ received by a worker is a \emph{consistent} but potentially outdated iterate. In the shared-memory setting, since we do not lock $\vx$ when a worker computes a stochastic gradient, $\widehat\vx^{(k)}$ may not be any iterate that ever exists in the memory but is a combination of a few iterates, i.e., the reading is \emph{inconsistent}; see Figure~\ref{fig:lockonoff} for an illustration. 

Suppose that the read of every coordinate is atomic. Then for each $i$, it must hold $\widehat x_i^{(k)}=x_i^{(k-j)}$ for some integer $j\ge0$. 
Let $I_j\coloneqq \left\{i\colon x^{(k-j)}_i=\widehat x^{(k)}_i\right\}$ and 
$\cI_j\coloneqq \cup_{l=0}^{j}I_l$ for each $j\ge 0$. 
Let $\tau_k = \min\left\{j\colon \cI_j=[n]\right\}$. Then $\cI_{\tau_k}=[n]$, and $\widehat\vx^{(k)}$ can be formed from 
$\{\vx^{(k-\tau_k)},\ldots,\vx^{(k)}\}$. 
By the definition of $\cI_l$, we have $\cI_{l-1}\subseteq \cI_l$, and thus
{
\begin{align}\label{eq:shared}
\widehat \vx^{(k)}
= & ~\vx^{(k)}\odot\vone_{\cI_0} + \sum_{l=1}^{\tau_k} \vx^{(k-l)}\odot(\vone_{\cI_l}-\vone_{\cI_{l-1}})\cr
= & ~ \vx^{(k)} - \vx^{(k)}\odot \vone_{\cI^c_0} + \sum_{l=1}^{\tau_k} \vx^{(k-l)}\odot(\vone_{\cI^c_{l-1}}-\vone_{\cI^c_l})\cr
= & ~ \vx^{(k)} - \sum_{l = 0}^{\tau_k-1}(\vx^{(k-l)}-\vx^{(k-l-1)})\odot\vone_{\cI^c_l},
\end{align}
}where we have used $\cI_{\tau_k}=[n]$, and $\vone_A$ represents the vector with one at each coordinate $i\in A$ and zero elsewhere.
The expression in \eqref{eq:shared} generalizes the relation for atomic lock-free updates in \cite{lian2015asynchronous, peng2016arock}.
It follows from \eqref{eq:shared} that 
\begin{align}\label{eq:x-hat1}
\|\widehat \vx^{(k)}- \vx^{(k)}\| \le &~\sum_{l=0}^{\tau_k-1} \big\|(\vx^{(k-l)}-\vx^{(k-l-1)})\odot \vone_{\cI^c_{l}}\big\|
\le \sum_{l=0}^{\tau_k-1} \big\| \vx^{(k-l)}-\vx^{(k-l-1)} \big\|,
\end{align}
and
\beq\label{eq:x-hat2}
\|\widehat \vx^{(k)}- \vx^{(k)}\|^2 \leq 
\tau_k\sum_{l=0}^{\tau_k-1} \|\vx^{(k-l)}-\vx^{(k-l-1)}\|^2.
\eeq
These relations are important in our analysis to handle asynchrony.

\section{Convergence for convex problems}\label{sec:cvx}

In this section, we analyze Algorithm~\ref{alg:async-adp-sgm} for convex problems. Throughout the analysis, we make the following assumptions. 
\begin{assumption}[convexity]\label{assump:cvx}
$F$ in \eqref{eq:stoc-prob} is convex, and $X$ is convex and compact.
\end{assumption}
Under Assumption~\ref{assump:cvx}, we define
$$D_\infty = \max_{\vx,\vy\in X}\|\vx-\vy\|_\infty.$$
\begin{assumption}[bounded gradient in expectation]\label{assump:bound-grad}
There is a finite number $G_1$ such that $\EE_\xi\|\tilde\nabla f(\vx,\xi)\|_1 \le G_1$, $\forall \vx\in X.$ 
\end{assumption}

\begin{assumption}[bounded gradient almost surely]\label{assump:bound-grad2}
	There is a finite number $G_\infty$ such that $ \|\tilde\nabla f(\vx;\xi)\|_\infty\le G_\infty,\, \forall \vx\in X,$ and almost surely for all $\xi.$
\end{assumption}

\begin{assumption}[unbiased gradient]\label{assump:unbiased}
$\vg^{(k)}$ is an unbiased estimate of a subgradient of $F$ at $\widehat\vx^{(k)}$ for each $k$, i.e., $\EE\big[\vg^{(k)}\,|\,\cH_k\big]\in\partial F(\widehat\vx^{(k)})$. 
\end{assumption}

We make a few remarks about the assumptions. The boundedness assumption on $X$ is required to analyze an adaptive SGM for convex problems in existing works, e.g., \cite{duchi2011adaptive, kingma2014adam, reddi2019convergence}. Assumption~\ref{assump:unbiased} is standard in the analysis of SGMs.  
It will hold in the distributed setting if data on all workers follow the same distribution and $\{\xi_i^{(k)}\}$ in \eqref{eq:vg-k} are sampled independently from the distribution. However, in the shared memory setting, the condition can only hold under certain ideal cases when asynchronous updates are performed. Roughly speaking, different realizations of $\{\xi_i^{(k)}\}$ in \eqref{eq:vg-k} can incur different cost of computing $\vg^{(k)}$ and thus affect the iteration counter $k$, i.e., $\widehat\vx^{(k)}$ can depend on $\{\xi_i^{(k)}\}$. Hence, the unbiased assumption can hold only if the cost of computing a stochastic subgradient is the same for any realization of $\xi$ and in addition the workers have the same computing power. \cite{leblond2018improved} addresses the independence issue by an ``after read'' approach. However, it could be computationally inefficient to first read the entire $\widehat\vx^{(k)}$ and then sample $\{\xi_i^{(k)}\}$ to compute $\vg^{(k)}$. This is also noticed in \cite{leblond2018improved}, which implements a different version of its analyzed method. The issue is also addressed in \cite{mania2017perturbed}, which essentially assumes sparsity of each sample gradient. Both of \cite{leblond2018improved, mania2017perturbed} require strong convexity on the objective. It is unclear whether the issue can be addressed for convex or non-convex cases.

We first establish a couple of lemmas that will be used to show the convergence rate of Algorithm~\ref{alg:async-adp-sgm} either with delay or without delay. Their proofs are given in the appendix.
\begin{lemma}\label{lem:ineq-from-opt}
Let $\{(\vx^{(k)},\vm^{(k)}, \widehat\vv^{(k)})\}$ be the sequence from Algorithm~\ref{alg:async-adp-sgm} with step size sequence $\{\alpha_k\}$. Under Assumption~\ref{assump:cvx}, 
it holds for any $t\ge 1$ and any $\vx\in X$ that
\begin{align}\label{eq:bd-crs-term}
& ~ \textstyle (1-\beta_1)\sum_{k=1}^t \left(\sum_{j=k}^t\alpha_j\beta_1^{j-k}\right)\left\langle \vx^{(k)} - \vx, \vg^{(k)} \right\rangle  
\le  \frac{D_\infty^2}{2}\|\sqrt{\widehat\vv^{(t)}}\|_1 + \frac{1}{2(1-\beta_1)^2}\sum_{k=1}^t\alpha_k^2\|\vm^{(k)}\|_{(\widehat\vv^{(k)})^{-\frac{1}{2}}}^2.
\end{align}
\end{lemma}

\begin{lemma}\label{lem:bd-m-v}
Let $\{(\vx^{(k)},\vm^{(k)}, \widehat\vv^{(k)})\}$ be the sequence from Algorithm~\ref{alg:async-adp-sgm}. Under Assumption~\ref{assump:bound-grad}, it holds 
\begin{equation}\label{eq:bd-m-v}
 \EE\|\vm^{(k)}\|_{(\widehat\vv^{(k)})^{-\frac{1}{2}}}^2\le \frac{G_1}{\sqrt{1-\beta_2}}. 
\end{equation}
\end{lemma}

\subsection{Convergence rate result for the case without delay}\label{sec:no.delay}

In this subsection, we use the previous two lemmas to show the convergence rate for the no-delay case, i.e., $\widehat \vx^{(k)}= \vx^{(k)},\forall \, k\ge1$ in \eqref{eq:vg-k}. 
Although the no-delay case is not our main focus, our results improve over existing ones about AMSGrad. 
\begin{theorem}[convex case without delay]\label{thm:cvx-no-delay}
Let $\{\vx^{(k)}\}$ be the sequence from Algorithm~\ref{alg:async-adp-sgm} with step size sequence $\{\alpha_k\}$.
Given an integer $K>0$, let $\bar\vx^{(K)} = \sum_{k=1}^K\frac{\sum_{j=k}^K\alpha_j\beta_1^{j-k}\vx^{(k)}}{\sum_{t=1}^K\left(\sum_{j=t}^K\alpha_j\beta_1^{j-t}\right)}$. Then under Assumptions~\ref{assump:cvx}-\ref{assump:unbiased}, we have the following results:
\begin{enumerate}
\item\label{thm:rate-cvx-no-delay}  If $\alpha_k = \frac{\alpha}{\sqrt K}, \,\forall k\ge 1$, for some $\alpha>0$, 
then
\begin{equation}\label{eq:rate-cvx-no-delay}
\EE\big[F(\bar\vx^{(K)}) - F^*\big] \le \frac{n D_\infty^2 G_\infty + \frac{\alpha^2}{(1-\beta_1)^2}\frac{G_1}{\sqrt{1-\beta_2}}}{2\alpha \sqrt{K}(1-\beta_1)}.
\end{equation}
\item\label{thm:rate-cvx-no-delay-adastp} If $\alpha_k = \frac{\alpha}{\sqrt k},\,\forall k\ge1$, for some $\alpha>0$, 
then
\begin{align}\label{eq:rate-cvx-no-delay-adastp}
\EE\big[F(\bar\vx^{(K)}) - F^*\big] 
\le  \frac{n D_\infty^2 G_\infty + \frac{\alpha^2(1+\log K)}{(1-\beta_1)^2}\frac{G_1}{\sqrt{1-\beta_2}}}{4\alpha (\sqrt{K+1}-1)(1-\beta_1)}.
\end{align}
\end{enumerate}
\end{theorem}

\begin{proof}
Taking expectation over both sides of \eqref{eq:bd-crs-term} and using Lemma~\ref{lem:bd-m-v}, we have
\begin{equation}\label{eq:ineq-cvx-case}
(1-\beta_1)\sum_{k=1}^t \left(\sum_{j=k}^t\alpha_j\beta_1^{j-k}\right)\EE\left\langle \vx^{(k)} - \vx, \vg^{(k)} \right\rangle  \le \frac{D_\infty^2}{2}\EE\|\sqrt{\widehat\vv^{(t)}}\|_1 + \frac{G_1}{2(1-\beta_1)^2\sqrt{1-\beta_2}}\sum_{k=1}^t\alpha_k^2.
\end{equation}
From Assumption~\ref{assump:unbiased}, we have \begin{equation}\label{eq:prob-crs-term}
\EE\big\langle \vx^{(k)} - \vx, \vg^{(k)} \big\rangle = \EE \big\langle \vx^{(k)} - \vx, \tilde\nabla F(\widehat \vx^{(k)})\big\rangle,
\end{equation}
 where $\tilde\nabla F(\widehat\vx^{(k)}) \in \partial F(\widehat\vx^{(k)})$. Hence, by the convexity of $F$ and $\widehat \vx^{(k)}= \vx^{(k)}$, it holds $\EE\big[F(\vx^{(k)})-F(\vx)\big]\le\EE\left\langle \vx^{(k)} - \vx, \vg^{(k)} \right\rangle$, and thus \eqref{eq:ineq-cvx-case} indicates
\begin{equation}\label{eq:fun-ineq-cvx-case}
(1-\beta_1)\sum_{k=1}^t \left(\sum_{j=k}^t\alpha_j\beta_1^{j-k}\right)\EE\big[F(\vx^{(k)})-F(\vx)\big]  \le \frac{D_\infty^2}{2}\EE\|\sqrt{\widehat\vv^{(t)}}\|_1 + \frac{G_1}{2(1-\beta_1)^2\sqrt{1-\beta_2}}\sum_{k=1}^t\alpha_k^2.
\end{equation}

Notice that $\sum_{j=k}^K\beta_1^{j-k}=\frac{1-\beta_1^{K-k+1}}{1-\beta_1}\in [1, \frac{1}{1-\beta_1})$. Hence, when $\alpha_k = \frac{\alpha}{\sqrt K}$, it holds 
$\sum_{k=1}^K\alpha_k^2 = \alpha^2$ and $\sum_{k=1}^K\left(\sum_{j=k}^K\alpha_j\beta_1^{j-k}\right) \ge \alpha \sqrt{K}.$
By the convexity of $F$, we have 
\begin{equation}\label{eq:cvx-ineq-F}
\sum_{t=1}^K\left(\sum_{j=t}^K\alpha_j\beta_1^{j-t}\right) F(\bar\vx^{(K)}) \le \sum_{k=1}^K\left(\sum_{j=k}^K\alpha_j\beta_1^{j-k}\right)F(\vx^{(k)}).
\end{equation} 
In addition, we have from \eqref{eq:rel-hat-vik} and Assumption~\ref{assump:bound-grad2} that $\EE\|\sqrt{\widehat\vv^{(t)}}\|_1 \le n G_\infty$. Now divide both sides of \eqref{eq:fun-ineq-cvx-case} by $\sum_{k=1}^K\left(\sum_{j=k}^K\alpha_j\beta_1^{j-k}\right)$, take $\vx$ as an optimal solution $\vx^*$, and let $t = K$. We obtain  the desired result in \eqref{eq:rate-cvx-no-delay}. 

When $\alpha_k = \frac{\alpha}{\sqrt k}$, it holds 
$$\sum_{k=1}^K\alpha_k^2 =\sum_{k=1}^K\frac{\alpha^2}{k}\le \alpha^2+ \int_{1}^K \frac{\alpha^2}{x}dx \le \alpha^2(1+\log K) $$ and $$\sum_{k=1}^K\left(\sum_{j=k}^K\alpha_j\beta_1^{j-k}\right)\ge \sum_{k=1}^K\frac{\alpha}{\sqrt{k}}\ge \alpha\int_{1}^{K+1} \frac{\alpha}{\sqrt{x}}dx \ge  2\alpha(\sqrt{K+1}-1).$$
Now utilizing the above bounds and \eqref{eq:cvx-ineq-F} and following the same arguments to show \eqref{eq:rate-cvx-no-delay}, 
we obtain  the desired result in \eqref{eq:rate-cvx-no-delay-adastp} and completes the proof. 
\end{proof}

\begin{remark} {Our rate is in the same order as that in \cite{reddi2019convergence}. However, by new analysis, we do not require an exponential or harmonic decaying sequence $\beta_{1,k}$ in computing $\vm$ vectors in \eqref{eq:update-m}. The decaying weight is required in \cite{reddi2019convergence} and also its follow-up works such as \cite{chen2018closing, luo2019adaptive}. Numerically, a fixed weight $\beta_1$ can give significantly better results, and indeed \cite{reddi2019convergence} uses $\beta_{1,k}=\beta_1,\forall k$ in all its experiments. The recent works \cite{chen2018convergence, zhou2018convergence} have also weakened the condition on decreasing $\beta_{1,k}$ for smooth non-convex cases. However, none of these works has dropped the assumption $\beta_1\le \sqrt{\beta_2}$ that is required by \cite{reddi2019convergence}. Our result does not need this condition. } 
\end{remark}

\subsection{Convergence rate result for the case with delay}\label{sec:delay}

In this subsection, we analyze the proposed algorithm when there is delay, i.e., $\tau_k>0$ in \eqref{eq:shared}. The delay naturally happens for asynchronous computing. It causes one main difficulty in bounding the expected objective error $\EE\big[F(\vx^{(k)})-F(\vx)\big]$ from using \eqref{eq:prob-crs-term}. That is because the difference $\tilde\nabla F(\widehat \vx^{(k)}) - \tilde\nabla F(\vx^{(k)})$ in general will not vanish as $\widehat \vx^{(k)}\neq \vx^{(k)}$ caused by the delay. Nevertheless, an ergodic sublinear convergence result can still be guaranteed under a few additional mild assumptions. 

\begin{assumption}[smoothness]\label{smoothness}
$F$ is $L$-smooth, i.e., $\|\nabla F(\vx)-\nabla F(\vy)\|\leq L \|\vx-\vy\|,\, \forall\,\vx,\vy\in \RR^n$.
\end{assumption}

\begin{assumption}[bounded staleness]\label{assump:bound-tau}
There is a finite integer $\tau$ such that $\tau_k\le \tau$ for all $k\ge1$.
\end{assumption}

With the master-worker set-up, we can measure the delay at the master, by counting the number of updates that are performed between two stochastic gradients computed by the same worker. Hence, by discarding too staled stochastic gradient, we can bound the staleness. In practice, we usually do not need to track the staleness. As mentioned in \cite{peng2019convergence, liu2020distributed}, the staleness is usually roughly equal to the number of processors, if all the processors have similar computing ability.

The next theorem gives a generic result for the case with delay.
\begin{theorem}\label{thm:rate-cvx-general}
Let $\{\vx^{(k)}\}$ be generated from Algorithm~\ref{alg:async-adp-sgm}. Under Assumptions~\ref{assump:cvx}, \ref{assump:bound-grad}, \ref{assump:unbiased} and \ref{smoothness}, we have that for any $\vx\in X$,
\begin{align}\label{eq:general-rate-cvx}
&(1-\beta_1)\sum_{k=1}^t\left(\sum_{j=k}^t\alpha_j\beta_1^{j-k}\right)\EE\big[F(\vx^{(k)})-F(\vx)\big] \\
\le &\frac{D_\infty^2}{2}\EE\|\sqrt{\widehat\vv^{(t)}}\|_1 + \frac{1}{2(1-\beta_1)^2}\frac{G_1}{\sqrt{1-\beta_2}}\sum_{k=1}^t\alpha_k^2 
+\frac{L(1-\beta_1)}{2}\sum_{k=1}^t\left(\sum_{j=k}^t\alpha_j\beta_1^{j-k}\right)\EE \|\vx^{(k)} - \widehat \vx^{(k)}\|^2.\nonumber
\end{align}
\end{theorem}
\begin{proof}
It follows from the convexity of $F$ that  
$$\big\langle \widehat \vx^{(k)} - \vx, \nabla F(\widehat \vx^{(k)})\big\rangle \ge F(\widehat \vx^{(k)})-F(\vx),$$
and in addition, the $L$-smoothness of $F$ implies   
$$\big\langle \vx^{(k)} - \widehat \vx^{(k)}, \nabla F(\widehat \vx^{(k)})\big\rangle \ge F(\vx^{(k)}) - F(\widehat \vx^{(k)}) - \frac{L}{2}\|\vx^{(k)} - \widehat \vx^{(k)}\|^2.$$
Plugging the above two inequalities into \eqref{eq:prob-crs-term}, we have
$$\EE\big\langle \vx^{(k)} - \vx, \vg^{(k)} \big\rangle \ge \EE\big[F(\vx^{(k)}) - F(\vx)\big] -\frac{L}{2}\EE \|\vx^{(k)} - \widehat \vx^{(k)}\|^2,$$
which together with \eqref{eq:ineq-cvx-case} gives the desired result in \eqref{eq:general-rate-cvx}.
\end{proof}

Note that if $\vx^{(k)} = \widehat \vx^{(k)},\forall\,k$, the inequality in \eqref{eq:general-rate-cvx} reduces to that in \eqref{eq:fun-ineq-cvx-case}. However, due to the asynchrony, we generally only have the relation in \eqref{eq:shared}. Therefore, the term about $\|\vx^{(k)} - \widehat \vx^{(k)}\|^2$ in \eqref{eq:general-rate-cvx} is not zero, and we need to bound it appropriately in order to establish the sublinear convergence. From \eqref{eq:x-hat2}, it suffices to bound $\|\vx^{(k+1)} - \vx^{(k)}\|^2$ for all $k$, which can be obtained by the following lemmas. The proofs of the lemmas are given in the appendix.

\begin{lemma}[non-expansiveness]\label{lem:nonexpansive}
Suppose $X=[a_1,b_1]\times\cdots \times [a_n,b_n]$ for some finite numbers $\{a_i\}$ and $\{b_i\}$. Let $\{\vx^{(k)}\}$ be generated from Algorithm~\ref{alg:async-adp-sgm}, and for any $i\in[n]$, we let $x_i^{(k+1)}=x_i^{(k)}$, if $\widehat v_i^{(k)}=0$. Then for any $k\ge1$, it holds
\begin{equation}\label{eq:diff-xk-1}
\|\vx^{(k+1)} - \vx^{(k)}\| \le \alpha_{k} \big\|\vm^{(k)}\oslash\sqrt{\widehat\vv^{(k)}}\big\|.
\end{equation}
\end{lemma}

\begin{remark}
Notice that when $\widehat v_i^{(k)}=0$, we must have $m_i^{(k)}=0$ and in this case, $x_i$ does not affect the objective of \eqref{eq:update-x}. Hence, when $X$ is separable and $\widehat v_i^{(k)}=0$, $x_i^{(k+1)}=x_i^{(k)}$ is one optimal choice for $x_i$, and thus such a setting  will not affect the optimality condition of \eqref{eq:update-x}.
\end{remark}

\begin{lemma}\label{lem:bd-mg-v}
Let $\{(\vg^{(k)}, \vm^{(k)}, \widehat\vv^{(k)})\}$ be generated from Algorithm~\ref{alg:async-adp-sgm}.
It holds for any $k\ge 1$ that
\begin{subequations}\label{eq:bd-m-divide-v}
\begin{align}
&\|\vg^{(j)}\oslash\sqrt{\widehat\vv^{(k)}}\|\le \frac{\sqrt{\|\vg^{(j)}\|_0}}{\sqrt{1-\beta_2}}, \forall j\le k, \label{eq:bd-m-divide-v-1}\\
&\|\vm^{(k)}\oslash\sqrt{\widehat\vv^{(k)}}\|\le  \sum_{j=1}^k(1-\beta_1)\beta_1^{k-j} \frac{\sqrt{\|\vg^{(j)}\|_0}}{\sqrt{1-\beta_2}}, \label{eq:bd-m-divide-v-2}\\
&\|\vm^{(k)}\oslash\sqrt{\widehat\vv^{(k)}}\|^2\le\frac{1-\beta_1}{1-\beta_2} \sum_{j=1}^k \beta_1^{k-j} \|\vg^{(j)}\|_0. \label{eq:bd-m-divide-v-3}
\end{align}
\end{subequations}
\end{lemma}

Applying the results in the above two lemmas to the inequality in \eqref{eq:general-rate-cvx}, we establish the sublinear convergence result of  Algorithm~\ref{alg:async-adp-sgm} as follows.

\begin{theorem}[convex case with delay]\label{thm:speed-up} 
Suppose Assumptions~\ref{assump:cvx} through \ref{assump:bound-tau} hold. Assume $X=[a_1,b_1]\times\cdots \times[a_n,b_n]$ for finite numbers $\{a_i\}$ and $\{b_i\}$. 
Let $\{\vx^{(k)}\}$ be generated from Algorithm~\ref{alg:async-adp-sgm}, and for any $i\in[n]$, we let $x_i^{(k+1)}=x_i^{(k)}$, if $\widehat v_i^{(k)}=0$. Given a positive integer $K$ and $\alpha>0$, let $\alpha_k= \frac{\alpha}{\sqrt K}$ for all $1\le k\le K$. Then 
{\begin{align}\label{eq:cvx-rate-const}
 \textstyle \EE\big[F(\bar\vx^{(K)}) - F^*\big] 
 \le \textstyle \frac{1}{2\alpha \sqrt{K}(1-\beta_1)}\left(n D_\infty^2 G_\infty +  \frac{\alpha^2G_1}{(1-\beta_1)^2\sqrt{1-\beta_2}} +\frac{\alpha^3L \tau^2  n }{\sqrt{K}(1-\beta_2)}\right),
 \end{align}}
where $\bar\vx^{(K)}$ is drawn from $\{\vx^{(k)}\}_{k=1}^K$ with
{
$$\textstyle \Prob(\bar\vx^{(K)}=\vx^{(k)})=\frac{\sum_{j=k}^K\alpha_j\beta_1^{j-k}}{\sum_{t=1}^K\left(\sum_{j=t}^K\alpha_j\beta_1^{j-t}\right)},\forall\,1\le k\le K.$$
}
\end{theorem}

\begin{proof}
By \eqref{eq:diff-xk-1} and \eqref{eq:bd-m-divide-v-3}, we have 
 $\EE\|\vx^{(k-l+1)} - \vx^{(k-l)}\|^2 \le \frac{ n  \alpha_{k-l}^2}{1-\beta_2}$. Since $\tau_k\le \tau,\forall\, k\ge1$, it  follows from \eqref{eq:x-hat2} that
 $$\EE \|\vx^{(k)} - \widehat \vx^{(k)}\|^2 \le \tau \sum_{l=1}^\tau\EE \|\vx^{(k-l+1)} - \vx^{(k-l)}\|^2.$$
 In addition, notice that 
 $$\sum_{k=1}^K\left(\sum_{j=k}^K\alpha_j\beta_1^{j-k}\right) \left(\tau\sum_{l=1}^\tau \alpha_{k-l}^2\right) \le \frac{\tau^2 \alpha^3  } {(1-\beta_1)\sqrt{K}}.$$
 Therefore
 $$\frac{L(1-\beta_1)}{2}\sum_{k=1}^K\left(\sum_{j=k}^K\alpha_j\beta_1^{j-k}\right)\EE \|\vx^{(k)} - \widehat \vx^{(k)}\|^2\le \frac{\alpha^3L \tau^2  n }{2\sqrt{K}(1-\beta_2)}.$$

Plugging the above inequality into \eqref{eq:general-rate-cvx} with $t=K$ and using $\sum_{k=1}^K\alpha_k^2 = \alpha^2$ give
 \begin{align*}
&~(1-\beta_1)\sum_{k=1}^K\left(\sum_{j=k}^K\alpha_j\beta_1^{j-k}\right)\EE\big[F(\vx^{(k)})-F(\vx)\big] \\
\le &~\frac{D_\infty^2}{2}\EE\|\sqrt{\widehat\vv^{(t)}}\|_1 + \frac{\alpha^2}{2(1-\beta_1)^2}\frac{G_1}{\sqrt{1-\beta_2}}+\frac{\alpha^3L \tau^2  n }{2\sqrt{K}(1-\beta_2)}.
\end{align*} 
 Now using the definition of $\bar\vx^{(K)}$ and noting $\sum_{k=1}^K\left(\sum_{j=k}^K\alpha_j\beta_1^{j-k}\right)\ge \alpha\sqrt{K}$, we obtain the result in \eqref{eq:cvx-rate-const} and complete the proof.
 \end{proof}

\begin{remark}[How delay affects convergence speed]\label{rm:rate-cvx}
Take $\alpha=O(1)$. 
Then 
\eqref{eq:cvx-rate-const} implies that the effect by the delay decreases at the rate of $
K^{\frac{1}{4}}$, namely, we can achieve nearly-linear speed-up if $\tau=o(
K^{\frac{1}{4}})$.  
\cite{peng2019convergence} shows that the delay, in expectation, equals the number of processors if all of them have the same computing power. Hence, in the ideal case, we can expect nearly-linear speed-up by using $o(K^{\frac{1}{4}})$ processors. 
However, notice that the convergence result of the asynchronous case requires stronger assumptions than that of the synchronous counterpart. 
In addition, 
 we set $\alpha_k= \frac{\alpha}{\sqrt K}$ for all $1\le k\le K$ in Theorem~\ref{thm:speed-up} for simplicity. A sublinear convergence result can also be established if $\alpha_k=\frac{\alpha}{\sqrt k}, \forall\, k$ by similar arguments as those in the proof of Theorem~\ref{thm:cvx-no-delay}.
\end{remark}

\section{Convergence for non-convex problems}\label{sec:ncvx}
\newcommand{\vGam}{{\mathbf{\Gamma}}}
\newcommand{\vPhi}{{\mathbf{\Phi}}}
\newtheorem{fact}{Fact}

In this section, we analyze Algorithm~\ref{alg:async-adp-sgm} for non-convex problems under Assumptions~\ref{assump:bound-grad2}--\ref{assump:bound-tau}. 
Due to the difficulty caused by nonconvexity, we assume $X=\RR^n$. Then the update in \eqref{eq:update-x} becomes 
\begin{equation}\label{eq:update-x-unc}
\textstyle \vx^{(k+1)}= \vx^{(k)} - \alpha_k \vm^{(k)}\oslash \sqrt{\widehat\vv^{(k)}}.
\end{equation}
When $X=\RR^n$, the gradient boundedness condition in Assumption~\ref{assump:bound-grad2} may not hold for deep learning problems. However, we are unable to relax this strong assumption. It is also made in all existing works that analyze adaptive SGMs for non-convex problems, e.g., \cite{zhou2018convergence, chen2018convergence}. On analyzing nonadaptive SGMs for non-convex problems, this assumption can be relaxed to a variance boundedness condition \cite{lian2015asynchronous, ghadimi2013stochastic}, which may not hold either for unconstrained deep learning problems but is weaker than the gradient boundedness condition.

Given a maximum number $K$ of iterations, we will assume, without loss of generality, $\widehat{v}^{(K)}_{i}> 0$ for all $i\in [n]$. Note that if $\widehat{v}^{(K)}_{i} = 0$ for some $i$, then $g_i^{(k)}=0$ for all $k\le K$, and in this case, $x_i$ never changes and can be simply viewed as a constant instead of a variable. 
We define an auxiliary sequence $\{ \widetilde{\vv}^{(k)} \}_{k=1}^{K}$ and gradient bounds as follows. These are only used in our analysis but not in the computation. 
\begin{definition}\label{def:Delta}
Given a positive integer $K$, let $\{\widehat\vv^{(k)}\}_{k=1}^{K}$ be computed from Algorithm~\ref{alg:async-adp-sgm}. For any $i\in [n]$, suppose $k_i\le K$ is the smallest number such that $\widehat{v}^{(k_i)}_{i}> 0$. Define $\{ \widetilde{\vv}^{(k)} \}_{k=1}^{K}$ as $\widetilde{v}^{(k)}_{i}= \max\{\widehat{v}^{(k)}_{i}, \, \widehat{v}^{(k_i)}_{i}\}$ for all $i\in [n]$ and all $k\in [K]$. Denote $\widetilde{\vV}^{(k)} = \Diag(\widetilde{\vv}^{(k)})$ for all $k\in [K].$
\end{definition}



\begin{definition}\label{def:fgbound}
Given a positive integer $K$, let $\{\vg^{(k)}\}_{k=1}^{K}$ and $\{ \vx^{(k)} \}_{k=1}^{K}$ be computed as in Algorithm~\ref{alg:async-adp-sgm}. We define $(\vGam(K),\vPhi(K))$ as: 
\beq\label{eq:fgbound}
\Gamma_i(K)= \max_{1\le k \le K}|g^{(k)}_i|,  \text{ and } \Phi_i(K) = \max_{1\le k \le K}|\nabla_i F(\vx^{(k)})|,  \ \forall\, i\in[n].
\eeq
We abbreviate the pair as $(\vGam,\vPhi)$ to hide the dependence on $K$, when it is clear from the context. 
\end{definition}

\begin{remark}\label{rm:vtilde}
We make a few remarks on $\{ \widetilde{\vv}^{(k)} \}_{k=1}^{K}$. (I) Assume $\widehat{v}^{(K)}_{i}> 0$ for all $i\in [n]$. Then each $\widetilde \vv^{(k)}$ is a positive vector, and $\widetilde \vv^{(k)}\ge \widetilde \vv^{(k-1)}$ still holds component-wisely; (II) $\vm^{(k)}\oslash\sqrt{\widetilde{\vv}^{(k)}}= \vm^{(k)}\oslash\sqrt{\widehat{\vv}^{(k)}}$ and $\vg^{(k)}\oslash\sqrt{\widetilde{\vv}^{(k)}}=\vg^{(k)}\oslash\sqrt{\widehat{\vv}^{(k)}}$ for all $k\le K$;  
and (III) $\widetilde{\vv}^{(K)} = \widehat{\vv}^{(K)}$ under the assumption $\widehat{\vv}^{(K)}>\vzero$. 
\end{remark}

\subsection{Preparatory lemmas}
In this subsection, we establish several lemmas. Their proofs are given in the appendix. The next lemma gives bounds on $\Gamma_i$ and $\Phi_i$ under Assumption~\ref{assump:bound-grad2}.
\begin{lemma}
\label{lem:indbnd}
Given a positive integer $K$, let $\{\widehat\vv^{(k)}\}_{k=1}^{K}$ and $\{\vm^{(k)}\}_{k=1}^{K}$ be generated from Algorithm~\ref{alg:async-adp-sgm}, and let $(\vGam,\vPhi)$ be given in Definition~\ref{def:fgbound}. We have for all $i\in[n],$ $| m^{(k)}_i|\leq  \Gamma_i,$ and 
$\widehat v^{(k)}_i \leq  \Gamma_i^2.$ 
Moreover, if Assumption~\ref{assump:bound-grad2} holds, then $\Gamma_i\le G_\infty$ almost surely, and $\Phi_i\le G_\infty $  for all $i\in[n].$
\end{lemma}

To analyze Algorithm~\ref{alg:async-adp-sgm} for non-convex problems, we follow the analytical framework of \cite{yang2016unified}. Let $\vx^{(0)}=\vx^{(1)}$, and we define an auxiliary sequence $\vz^{(k)}$ as follows: 
\begin{equation}\label{defzk}
\vz^{(k)}=\vx^{(k)}+\frac{\beta_1}{1-\beta_1}(\vx^{(k)}-\vx^{(k-1)})=\frac{1}{1-\beta_1}\vx^{(k)}-\frac{\beta_1}{1-\beta_1}\vx^{(k-1)},\, \forall\, k\ge1.
\end{equation}

The following lemma is from Lemma A.3 of \cite{zhou2018convergence}. It shows that $\vz^{(k+1)}-\vz^{(k)}$ can be represented in two different ways. However, due to typos in the original proof, we provide a complete proof in the appendix for the convenience of the readers.
\begin{lemma}\label{lem:zz}
Let $\vz^{(k)}$ be defined as in \eqref{defzk} and $\widetilde\vV^{(k)}$ in Definition~\ref{def:Delta}. We have 
\begin{equation}\label{zzatone}
\vz^{(2)}-\vz^{(1)}=-\alpha_1(\widetilde\vV^{(1)})^{-\frac{1}{2}}\vg^{(1)},
\end{equation}
and for $k= 2,\ldots,K$,
\begin{align}\label{zzandm}
\vz^{(k+1)}-\vz^{(k)}=&~\frac{\beta_1}{1-\beta_1}\left[\alpha_{k-1}(\widetilde\vV^{(k-1)})^{-\frac{1}{2}}-\alpha_k(\widetilde\vV^{(k)})^{-\frac{1}{2}}\right]\vm^{(k-1)}-\alpha_k(\widetilde\vV^{(k)})^{-\frac{1}{2}}\vg^{(k)},\\
=&~\frac{\beta_1}{1-\beta_1}\left[\vI - \alpha_k(\widetilde\vV^{(k)})^{-\frac{1}{2}}  \alpha_{k-1}^{-1}(\widetilde\vV^{(k-1)})^{\frac{1}{2}}  \right](\vx^{(k-1)}-\vx^{(k)})-\alpha_k(\widetilde\vV^{(k)})^{-\frac{1}{2}}\vg^{(k)}.\label{zzandxx}
\end{align} 
\end{lemma}

\begin{lemma}\label{lem:terms1}
Let $\{\vx^{(k)}\}$ be from Algorithm~\ref{alg:async-adp-sgm}, $\{\vz^{(k)}\}$ defined as in \eqref{defzk}, and $\{\widetilde\vv^{(k)}\}$ in Definition~\ref{def:Delta}. Also, let $\{\alpha_k\}$ be a non-increasing positive sequence. For $k= 2,\ldots,K,$ we have
\begin{align}\label{eq:terms1}
&~\nabla F(\vx^{(k)})^{\top} (\vz^{(k+1)}-\vz^{(k)})\\
\leq &~ -\nabla F(\vx^{(k)})^{\top} \alpha_{k-1}(\widetilde\vV^{(k-1)})^{-\frac{1}{2}}\vg^{(k)} + \frac{1}{1-\beta_1}\sum_{i=1}^{n}  \Gamma_i\Phi_i \left[\alpha_{k-1}(\widetilde v^{(k-1)}_i)^{-\frac{1}{2}}-\alpha_k(\widetilde v^{(k)}_i)^{-\frac{1}{2}} \right].\nonumber
\end{align}
\end{lemma}

The next two lemmas are directly from \cite{zhou2018convergence}. Although the original results are for $k\ge2$, they trivially hold when $k=1$.

\begin{lemma}[Lemma A.4 of \cite{zhou2018convergence}]\label{lem:zznormbound}
Let $\{\vz^{(k)}\}$ be defined as in \eqref{defzk}, and let $\{\alpha_k\}_{k\geq 1}$ be a non-increasing positive sequence. For $k\ge 1$, we have
\begin{equation}\label{eq:zznormbound}
\|\vz^{(k+1)}-\vz^{(k)}\|\leq \frac{\beta_1}{1-\beta_1} \|\vx^{(k-1)}-\vx^{(k)}\|+\|\alpha_k(\widetilde\vV^{(k)})^{-\frac{1}{2}}\vg^{(k)}\|.
\end{equation}
\end{lemma}
\begin{lemma}[Lemma A.5 of \cite{zhou2018convergence}]\label{lem:zxlip}
Let $\{\vz^{(k)}\}$ be defined as in \eqref{defzk} and $\widetilde\vV^{(k)}$ in Definition~\ref{def:Delta}.  Under Assumption~\ref{smoothness}, we have
\begin{equation}\label{eq:zxlip}
\|\nabla F(\vz^{(k)})-\nabla F(\vx^{(k)})\|\leq \frac{L\beta_1 }{1-\beta_1} \|\vx^{(k-1)}-\vx^{(k)}\|, \forall\, k\ge 1.
\end{equation}
\end{lemma}

We still need the following lemma to show our main convergence result for the non-convex case.

\begin{lemma}\label{lem:terms23}
Let $\{\vz^{(k)}\}$ be defined as in \eqref{defzk} and $\widetilde\vV^{(k)}$ in Definition~\ref{def:Delta}. Also, let $\{\alpha_k\}$ be a non-increasing positive sequence. Under Assumption~\ref{smoothness}, we have that for all $k\ge1$,
\begin{equation}\label{eq:terms2}
\left(\nabla F(\vz^{(k)})-\nabla F(\vx^{(k)})\right)^{\top} (\vz^{(k+1)}-\vz^{(k)})\leq \frac{3L\beta_1^2}{2(1-\beta_1)^2} \|\vx^{(k-1)}-\vx^{(k)}\|^2+\frac{L}{2}\|\alpha_k(\widetilde\vV^{(k)})^{-\frac{1}{2}}\vg^{(k)}\|^2,
\end{equation}
and 
\begin{equation}\label{eq:terms3}
\|\vz^{(k+1)}-\vz^{(k)}\|^2\leq \frac{4\beta_1^2}{(1-\beta_1)^2} \|\vx^{(k-1)}-\vx^{(k)}\|^2+\frac{4}{3}\|\alpha_k(\widetilde\vV^{(k)})^{-\frac{1}{2}}\vg^{(k)}\|^2.
\end{equation}
\end{lemma}

\subsection{Convergence rate results}

By the lemmas established in the previous subsection, we are ready to show the convergence result of Algorithm~\ref{alg:async-adp-sgm} for non-convex problems. The next theorem gives the convergence rate without specifying the learning rate $\{\alpha_k\}$.

\begin{theorem}\label{thm:ncv}
Given an integer $K\ge2$, let $\{\vx^{(k)}\}_{k=1}^K$ and $\{\widehat\vv^{(k)}\}_{k=1}^K$ be generated from Algorithm~\ref{alg:async-adp-sgm} with a non-increasing positive sequence $\{\alpha_k\}_{k=1}^{K}$. Suppose $\widehat\vv^{(K)}>\vzero$.  
Suppose that there is a constant $C_F$ such that $|F(\vx)| \le C_F,\, \forall\,\vx$. Let $2\le k_0 \le K$ and $\bar\vx^{(k_0,K)}$ be drawn from  $\{\vx^{(k)}\}_{k=k_0}^{K}$ with probability 
\begin{equation}\label{eq:xbar-k0-K}
\textstyle \Prob(\bar\vx^{(k_0,K)}= \vx^{(k)}) = \frac{\alpha_{k-1}}{\sum_{j=k_0}^{K}\alpha_{j-1}},\,\forall\, k=k_0,\ldots, K.
\end{equation}
Then under Assumptions~\ref{assump:bound-grad2} through \ref{assump:bound-tau}, it holds
\begin{align}\label{eq:telescope3-k0}
&~\EE\|\nabla F(\bar\vx^{(k_0,K)})\|^2\\
\leq &~  \frac{  G_{\infty}^3\EE \|(\widetilde \vv^{(k_0-1)})^{-\frac{1}{2}}\|_1 }{1-\beta_1}   
 \frac{\alpha_{k_0-1}}{\sum_{k=k_0}^{K}\alpha_{k-1}}+\frac{ 2C_F G_{\infty}}{\sum_{k=k_0}^{K}\alpha_{k-1}}+ \frac{7  n  L G_{\infty}}{6 (1-\beta_2)}\frac{ \sum_{k=k_0}^{K}\alpha_{k}^2}{\sum_{k=k_0}^{K}\alpha_{k-1}} \cr
 &+ \frac{7  n  L G_{\infty}\beta_1^2}{2 (1-\beta_2)(1-\beta_1)^2} \cdot\frac{ \sum_{k=k_0}^{K}\alpha_{k-1}^2}{\sum_{k=k_0}^{K}\alpha_{k-1}} + \frac{ \sqrt{ n } L G_{\infty}  \overset{K}{ \underset{k=k_0}\sum}\alpha_{k-1}\overset{\tau}{\underset{j=1}\sum}\alpha_{k-j}\sqrt{\EE\|(\widetilde\vV^{(k-1)})^{-\frac{1}{2}}\nabla F(\vx^{(k)})\|^2}}{\sqrt{1-\beta_2}\sum_{t=k_0}^{K}\alpha_{t-1}}.\nonumber
\end{align}
\end{theorem}
\begin{proof}
By the $L$-smoothness of $F,$ it follows that
\begin{align}\label{eq:break}
F(\vz^{(k+1)})\leq &~ F(\vz^{(k)})  + \nabla F(\vz^{(k)})^{\top} (\vz^{(k+1)}-\vz^{(k)}) + \frac{L}{2}\|\vz^{(k+1)}-\vz^{(k)}\|^2\nonumber\\
= &~F(\vz^{(k)})  + \nabla F(\vx^{(k)})^{\top} (\vz^{(k+1)}-\vz^{(k)}) \\
&~+ \left(\nabla F(\vz^{(k)})-\nabla F(\vx^{(k)})\right)^{\top} (\vz^{(k+1)}-\vz^{(k)}) + \frac{L}{2}\|\vz^{(k+1)}-\vz^{(k)}\|^2. \nonumber
\end{align}


For $2\le k\le K$, we substitute \eqref{eq:terms1}, \eqref{eq:terms2} and \eqref{eq:terms3} into \eqref{eq:break} and rearrange terms to have  
\begin{align}\label{eq:boundk2a}
&~F(\vz^{(k+1)}) + \frac{\sum_{i=1}^{n}  \Gamma_i\Phi_i \alpha_{k}(\widetilde v^{(k)}_i)^{-\frac{1}{2}}}{1-\beta_1} -\Bigg( F(\vz^{(k)}) + \frac{\sum_{i=1}^{n}  \Gamma_i\Phi_i \alpha_{k-1}(\widetilde v^{(k-1)}_i)^{-\frac{1}{2}}}{1-\beta_1} \Bigg)\nonumber\\
\leq &~ -\nabla F(\vx^{(k)})^{\top} \alpha_{k-1}(\widetilde\vV^{(k-1)})^{-\frac{1}{2}}\vg^{(k)} + \frac{7L\beta_1^2}{2(1-\beta_1)^2} \|\vx^{(k-1)}-\vx^{(k)}\|^2+\frac{7L}{6}\|\alpha_k(\widetilde\vV^{(k)})^{-\frac{1}{2}}\vg^{(k)}\|^2\nonumber\\
= & ~-\nabla F(\vx^{(k)})^{\top} \alpha_{k-1}(\widetilde\vV^{(k-1)})^{-\frac{1}{2}}\vg^{(k)} + \frac{7L\beta_1^2}{2(1-\beta_1)^2} \|\alpha_{k-1}\vm^{(k-1)}\oslash\sqrt{\widehat{\vv}^{(k-1)}} \|^2\nonumber\\
 & ~+\frac{7L}{6}\|\alpha_k(\widetilde\vV^{(k)})^{-\frac{1}{2}}\vg^{(k)}\|^2.
\end{align}
From Assumption~\ref{assump:unbiased}, we have $\EE\left[\vg^{(k)}\,|\,\cH_k\right]=\nabla F(\widehat \vx^{(k)})$, and thus taking the expectation on both sides of \eqref{eq:boundk2a} gives 
\begin{align}\label{eq:boundk2b}
&~\EE\Bigg[F(\vz^{(k+1)}) +   \frac{G_\infty^2 \alpha_{k}\|(\widetilde \vv^{(k)})^{-\frac{1}{2}}\|_1}{1-\beta_1} -  F(\vz^{(k)}) -\frac{G_\infty^2 \alpha_{k-1} \|(\widetilde \vv^{(k-1)} )^{-\frac{1}{2}}\|_1}{1-\beta_1} \Bigg]   \nonumber\\
\leq &~\EE\Bigg[F(\vz^{(k+1)}) +   \frac{\sum_{i=1}^{n}  \Gamma_i\Phi_i \alpha_{k}(\widetilde v^{(k)}_i)^{-\frac{1}{2}}}{1-\beta_1} -  F(\vz^{(k)}) -\frac{\sum_{i=1}^{n}  \Gamma_i\Phi_i \alpha_{k-1}(\widetilde v^{(k-1)}_i)^{-\frac{1}{2}}}{1-\beta_1} \Bigg]   \nonumber\\
\leq &~ \EE\Bigg[\frac{7L}{6}\|\alpha_k(\widetilde\vV^{(k)})^{-\frac{1}{2}}\vg^{(k)}\|^2 -\nabla F(\vx^{(k)})^{\top} \alpha_{k-1}(\widetilde\vV^{(k-1)})^{-\frac{1}{2}}\nabla F(\widehat \vx^{(k)}) \nonumber\\
&~+ \frac{7L\beta_1^2}{2(1-\beta_1)^2} \|\alpha_{k-1}\vm^{(k-1)}\oslash\sqrt{\widehat{\vv}^{(k-1)}} \|^2 \Bigg] \nonumber\\
= &~\EE\left[\frac{7L}{6}\|\alpha_k(\widetilde\vV^{(k)})^{-\frac{1}{2}}\vg^{(k)}\|^2 + \frac{7L\beta_1^2}{2(1-\beta_1)^2} \|\alpha_{k-1}\vm^{(k-1)}\oslash\sqrt{\widehat{\vv}^{(k-1)}} \|^2 \right]\nonumber\\
&~+\EE\left[  \nabla F(\vx^{(k)})^{\top} \alpha_{k-1}(\widetilde\vV^{(k-1)})^{-\frac{1}{2}}\left(\nabla F(\vx^{(k)}) -\nabla F(\widehat \vx^{(k)})\right)\right]\nonumber\\
&~-\EE\left[\nabla F(\vx^{(k)})^{\top} \alpha_{k-1}(\widetilde\vV^{(k-1)})^{-\frac{1}{2}}\nabla F(\vx^{(k)})
\right],
\end{align}
where the first inequality follows from $ \alpha_{k}(\widetilde v^{(k)}_i)^{-\frac{1}{2}}-\alpha_{k-1}(\widetilde v^{(k-1)}_i)^{-\frac{1}{2}}\le 0$ for all $i\in[n]$ and Lemma~\ref{lem:indbnd}. 
By the Cauchy-Schwarz inequality, the smoothness of $F$,  Assumption~\ref{assump:bound-tau}, and equations \eqref{eq:x-hat1} and \eqref{eq:update-x-unc}, we have
\begin{align*}
&~\nabla F(\vx^{(k)})^{\top} (\widetilde\vV^{(k-1)})^{-\frac{1}{2}}\left(\nabla F(\vx^{(k)}) -\nabla F(\widehat \vx^{(k)})\right) \\
 \le &~ \|(\widetilde\vV^{(k-1)})^{-\frac{1}{2}}\nabla F(\vx^{(k)})\| \cdot \|\nabla F(\vx^{(k)}) -\nabla F(\widehat \vx^{(k)})\|\\
 \le &~ L \|(\widetilde\vV^{(k-1)})^{-\frac{1}{2}}\nabla F(\vx^{(k)})\|\left(\sum_{j=1}^\tau \|\alpha_{k-j}\vm^{(k-j)}\oslash\sqrt{\widehat{\vv}^{(k-j)}} \|.\right)
\end{align*}
In addition, by Definition~\ref{def:Delta} and Lemma~\ref{lem:indbnd}, 
it follows
\begin{align*} 
\EE\left[\nabla F(\vx^{(k)})^{\top} (\widetilde\vV^{(k-1)})^{-\frac{1}{2}}\nabla F(\vx^{(k)})\right]
&\geq \EE\left[\|\widetilde\vv^{(k-1)}\|_{\infty}^{-\frac{1}{2}} \|\nabla F(\vx^{(k)})\|^2\right]  
\ge G_{\infty}^{-1}\EE\|\nabla F(\vx^{(k)})\|^2.
\end{align*}
Substituting the above two inequalities into \eqref{eq:boundk2b},  we have 
\begin{align}
& \EE\Bigg[F(\vz^{(k+1)}) +   \frac{G_\infty^2 \alpha_{k}\|(\widetilde \vv^{(k)})^{-\frac{1}{2}}\|_1}{1-\beta_1} -  F(\vz^{(k)}) -\frac{G_\infty^2 \alpha_{k-1} \|(\widetilde \vv^{(k-1)}_i )^{-\frac{1}{2}}\|_1}{1-\beta_1} \Bigg]   \nonumber\\
\leq& \EE\left[\frac{7L}{6}\|\alpha_k(\widetilde\vV^{(k)})^{-\frac{1}{2}}\vg^{(k)}\|^2 -  \frac{\alpha_{k-1}}{G_{\infty}} \|\nabla F(\vx^{(k)})\|^2
+ \frac{7L\beta_1^2}{2(1-\beta_1)^2}  \|\alpha_{k-1}\vm^{(k-1)}\oslash\sqrt{\widehat{\vv}^{(k-1)}} \|^2
\right] \nonumber\\
&+\EE\left[ \alpha_{k-1} L \|(\widetilde\vV^{(k-1)})^{-\frac{1}{2}}\nabla F(\vx^{(k)})\| \left(  \sum_{j=1}^\tau \|\alpha_{k-j}\vm^{(k-j)}\oslash\sqrt{\widehat{\vv}^{(k-j)}} \|\right)\right].\label{eq:boundk2c}
\end{align}
For any $2\le k_0\le K$, summing \eqref{eq:boundk2c} over $k=k_0,\ldots,K$ and using the condition $|F(\vx)| \le C_F,\, \forall\,\vx$, 
we have
\begin{align}
&  G_{\infty}^{-1} \sum_{k=k_0}^{K} \alpha_{k-1}\EE\|\nabla F(\vx^{(k)})\|^2\nonumber\\
\leq &~ 
2C_F+ \frac{G_\infty^2\alpha_{k_0-1}\EE\|(\widetilde \vv^{(k_0-1)})^{-\frac{1}{2}}\|_1}{1-\beta_1}   \nonumber\\
&+ \frac{7L}{6} \sum_{k=k_0}^{K}\EE\|\alpha_k\vg^{(k)}\oslash\sqrt{\widehat{\vv}^{(k)}}\|^2 + \frac{7L\beta_1^2}{2(1-\beta_1)^2}\sum_{k=k_0}^{K}\EE \|\alpha_{k-1}\vm^{(k-1)}\oslash\sqrt{\widehat{\vv}^{(k-1)}} \|^2 \nonumber\\
&+ L  \sum_{k=k_0}^{K}\alpha_{k-1} \EE\left[\|(\widetilde\vV^{(k-1)})^{-\frac{1}{2}}\nabla F(\vx^{(k)})\|   \left(\sum_{j=1}^\tau \|\alpha_{k-j}\vm^{(k-j)}\oslash\sqrt{\widehat{\vv}^{(k-j)}} \| \right)\right].\label{eq:telescope}
\end{align}
Now use \eqref{eq:bd-m-divide-v-2} of Lemma~\ref{lem:bd-mg-v} in the above inequality to have
\begin{align}
&~  G_{\infty}^{-1} \sum_{k=k_0}^{K} \alpha_{k-1}\EE\|\nabla F(\vx^{(k)})\|^2\nonumber\\
\leq &~  
2C_F+ \frac{G_\infty^2\alpha_{k_0-1}\EE\|(\widetilde \vv^{(k_0-1)})^{-\frac{1}{2}}\|_1}{1-\beta_1} + \frac{7  n  L}{6(1-\beta_2)} \sum_{k=k_0}^{K}\alpha_k^2 
  + \frac{7  n  L\beta_1^2}{2(1-\beta_2)(1-\beta_1)^2}\sum_{k=k_0}^{K}\alpha_{k-1}^2 \nonumber\\
&+ \frac{  L }{\sqrt{1-\beta_2}}\sum_{k=k_0}^{K}\alpha_{k-1}\sum_{j=1}^\tau\alpha_{k-j}\sum_{l=1}^{k-j}(1-\beta_1)\beta_1^{k-j-l}\EE\left[\|(\widetilde\vV^{(k-1)})^{-\frac{1}{2}}\nabla F(\vx^{(k)})\| \sqrt{\|\vg^{(l)}\|_0}\right]   \nonumber\\
\leq &~  
2C_F+ \frac{G_\infty^2\alpha_{k_0-1}\EE\|(\widetilde \vv^{(k_0-1)})^{-\frac{1}{2}}\|_1}{1-\beta_1}  + \frac{7  n  L}{6(1-\beta_2)} \sum_{k=k_0}^{K}\alpha_k^2 
  + \frac{7  n  L\beta_1^2}{2(1-\beta_2)(1-\beta_1)^2}\sum_{k=k_0}^{K}\alpha_{k-1}^2 \nonumber\\
&+ \frac{ \sqrt{ n } L }{\sqrt{1-\beta_2}}\sum_{k=k_0}^{K}\alpha_{k-1}\sqrt{\EE\left[\|(\widetilde\vV^{(k-1)})^{-\frac{1}{2}}\nabla F(\vx^{(k)})\|^2\right]}   \sum_{j=1}^\tau\alpha_{k-j},\label{eq:telescope2}
\end{align}
where the last inequality is by Cauchy-Schwarz inequality.
Dividing $G_{\infty}^{-1} \sum_{k=k_0}^{K}\alpha_{k-1}$ on both sides of \eqref{eq:telescope2} and using the definition of $\bar\vx^{(k_0,K)},$ we obtain the desired result in  \eqref{eq:telescope3-k0}. 
\end{proof}


Below we specify the choices of $\{\alpha_k\}$ and show the sublinear convergence.
\begin{theorem}\label{thm:NCV}
Given an integer $K\ge2$, let $\{\vx^{(k)}\}_{k=1}^K$ and $\{\widehat\vv^{(k)}\}_{k=1}^K$ be generated from Algorithm~\ref{alg:async-adp-sgm} with a non-increasing positive sequence $\{\alpha_k\}_{k=1}^{K}$. For some $2\le k_0\le K$, let $\bar\vx^{(k_0,K)}$ be drawn from  $\{\vx^{(k)}\}_{k=k_0}^{K}$ according to \eqref{eq:xbar-k0-K}.
Suppose Assumptions~\ref{assump:bound-grad2} through \ref{assump:bound-tau}
 hold. In addition, assume $\widehat\vv^{(K)}>\vzero$ 
Moreover, suppose that there are positive constants $C_F$ and $c$ such that $|F(\vx)| \le C_F,\, \forall\,\vx$, and $\widetilde v_i^{(k_0-1)}\ge c^2, \forall i\in [n]$ hold almost surely. The following results hold:
\begin{enumerate}
    \item\label{cor:ncv1} If $\alpha_k = \frac{\alpha}{\sqrt{K-k_0+1}}$ 
    for all $k$ and some constant $\alpha>0$, then 
\begin{equation}\label{eq:ncvconstant2}
\textstyle \EE\|\nabla F(\bar\vx^{(k_0,K)})\|^2\le C_1 + \frac{C_2}{c}\left(\sqrt{C_1}+\frac{C_2}{c}\right),
\end{equation}
where $C_2 =  \frac{ \alpha\tau\sqrt{ n } L G_{\infty}  }{\sqrt{1-\beta_2}\sqrt{K-k_0+1}}$, and
{
\begin{align*}
\textstyle C_1 = &~  \textstyle \frac{  G_{\infty}^3\EE \|(\widetilde \vv^{(k_0-1)})^{-\frac{1}{2}}\|_1 }{(1-\beta_1)(K-k_0+1)} +\frac{ 2C_F G_{\infty}}{\alpha\sqrt{K-k_0+1}}  \textstyle+\frac{7  n  L G_{\infty}(1-2\beta_1+4\beta_1^2)}{6 (1-\beta_2)(1-\beta_1)^2} \frac{\alpha}{\sqrt{K-k_0+1}}.
\end{align*}
} 
\vspace{-.2cm}
    \item\label{cor:ncv2} Let $k_0 = \frac{K}{2} \ge \tau+2$. If $\alpha_k = \frac{\alpha}{\sqrt{k}}$ for all $k\ge1$, 
    then \eqref{eq:ncvconstant2} holds, where $C_2 =  \textstyle\frac{2\sqrt{2} \alpha\tau\sqrt{ n } L G_{\infty}  }{\sqrt{K}\sqrt{1-\beta_2}}$, and
\begin{align*}
C_1 = &~   \textstyle \frac{  G_{\infty}^3\EE \|(\widetilde \vv^{(k_0-1)})^{-\frac{1}{2}}\|_1 }{(2-\sqrt{2})(1-\beta_1)\sqrt{K}\sqrt{K/2-1}}   
 \textstyle+\frac{ 2C_F G_{\infty}}{(2-\sqrt{2})\alpha\sqrt{K}} 
 + \textstyle \frac{7  n  L G_{\infty}}{6 (1-\beta_2)}\frac{ \alpha \log 4}{(2-\sqrt{2})\sqrt{K}} + \frac{7  n  L G_{\infty}\beta_1^2}{2 (1-\beta_2)(1-\beta_1)^2}  \frac{ \alpha(1+\log 3)}{(2-\sqrt{2})\sqrt{K}}.
\end{align*}
  \end{enumerate}
\end{theorem}

\begin{proof}
\textbf{Setting 1:}~~When $\alpha_k = \frac{\alpha}{\sqrt{K-k_0+1}},\,\forall 1\le k\le K$, we plug $\alpha_k$ into \eqref{eq:telescope3-k0} and have
\begin{align}\label{eq:ncvconstant2-tmp}
\EE\|\nabla F(\bar\vx^{(k_0,K)})\|^2 \le C_1 + 
 \frac{C_2}{K-k_0+1}\sum_{k=k_0}^K \sqrt{\EE\left[\|(\widetilde\vV^{(k-1)})^{-\frac{1}{2}}\nabla F(\vx^{(k)})\|^2\right]}.
\end{align}
Since $\widetilde v_i^{(k_0-1)}\ge c^2, \forall i\in [n]$ almost surely and $\widetilde \vv^{(k+1)}\ge \widetilde \vv^{(k)},\forall k\ge1$, it holds
\begin{align*}
    \frac{1}{K-k_0+1}\sum_{k=k_0}^K \sqrt{\EE\left[\|(\widetilde\vV^{(k-1)})^{-\frac{1}{2}}\nabla F(\vx^{(k)})\|^2\right]}
    \le & \frac{1}{c(K-k_0+1)}\sum_{k=k_0}^K \sqrt{\EE\|\nabla F(\vx^{(k)})\|^2} \\
    \le & \frac{1}{c}\sqrt{\EE \|\nabla F(\bar\vx^{(k_0,K)})\|^2},
\end{align*}
where the last inequality follows from Jensen's inequality.  
Hence, plugging the above inequality into \eqref{eq:ncvconstant2-tmp} yields
\begin{equation}\label{eq:bd-gradF-x-k0-K}
\EE \|\nabla F(\bar\vx^{(k_0,K)})\|^2 \le C_1 + \frac{C_2}{c} \sqrt{\EE \|\nabla F(\bar\vx^{(k_0,K)})\|^2},
\end{equation}
which implies $\sqrt{\EE \|\nabla F(\bar\vx^{(k_0,K)})\|^2} \le \sqrt{C_1}+\frac{C_2}{c}$. Therefore, we have the desired result from \eqref{eq:bd-gradF-x-k0-K}. 

\vspace{0.2cm}
\noindent\textbf{Setting 2:}~~When $\alpha_k = \frac{\alpha}{\sqrt{k}},\,\forall 1\le k\le K$, we have
$$\sum_{k=k_0}^{K}\alpha_{k-1}= \sum_{k=k_0}^{K}\frac{\alpha}{\sqrt{k-1}} \ge\int_{k_0-1}^K \frac{\alpha}{\sqrt{x}} dx = 2\alpha (\sqrt{K}-\sqrt{k_0-1})\ge (2-\sqrt{2})\alpha \sqrt{K},$$
and
$$\sum_{k=k_0}^{K}\alpha_k^2 = \sum_{k=k_0}^{K}\frac{\alpha^2}{k}\le \int_{k_0-1}^K \frac{\alpha^2}{x}dx = \alpha^2 \log\frac{K}{k_0-1}=\alpha^2 \log\frac{1}{\frac{1}{2}-\frac{1}{K}}\le \alpha^2 \log 4.$$
Similarly,
$$\sum_{k=k_0}^{K}\alpha_{k-1}^2 = \sum_{k=k_0}^{K}\frac{\alpha^2}{k-1}\le \frac{\alpha^2}{k_0-1} + \int_{k_0-1}^{K-1} \frac{\alpha^2}{x}dx \le \alpha^2 + \alpha^2\log\frac{K-1}{k_0-1} \le \alpha^2(1+\log 3),
$$
and for all $k\ge k_0\ge\tau+2$, 
$$\sum_{j=1}^\tau\alpha_{k-j} \le \sum_{j=1}^\tau\alpha_{k_0-j} = \sum_{j=1}^\tau\frac{\alpha}{\sqrt{k_0-j}} \le \int_{k_0-\tau-1}^{k_0-1} \frac{\alpha}{\sqrt{x}}dx = 2\alpha (\sqrt{k_0-1}-\sqrt{k_0-\tau-1}) \le \frac{2\alpha\tau}{\sqrt{k_0}}. $$
Therefore, plugging $\alpha_k = \frac{\alpha}{\sqrt{k}},\,\forall 1\le k\le K$ into \eqref{eq:telescope3-k0} and using the above inequalities, we have
\begin{align}\label{eq:telescope3-k0-ncvx-2}
\EE\|\nabla F(\bar\vx^{(k_0,K)})\|^2
&\leq C_1+ \frac{C_2 }{\sum_{k=k_0}^{K}\alpha_{k-1}}  \sum_{k=k_0}^{K}\alpha_{k-1}\sqrt{\EE\left[\|(\widetilde\vV^{(k-1)})^{-\frac{1}{2}}\nabla F(\vx^{(k)})\|^2\right]}.
\end{align}
Notice $\widetilde v_i^{(k_0-1)}\ge c^2, \forall i\in [n]$ almost surely, and also use the definition of $\bar\vx^{(k_0,K)}$ in \eqref{eq:xbar-k0-K}. We have, by Jensen's inequality,
$$\frac{1}{\sum_{k=k_0}^{K}\alpha_{k-1}}\sum_{k=k_0}^{K}\alpha_{k-1}\sqrt{\EE\left[\|(\widetilde\vV^{(k-1)})^{-\frac{1}{2}}\nabla F(\vx^{(k)})\|^2\right]} \le \frac{1}{c} \sqrt{\EE\|\nabla F(\bar\vx^{(k_0,K)})\|^2}.$$
Now by the same arguments as those in the proof of Setting 1, we obtain the desired result.
\end{proof}

\begin{remark}\label{rmk:ncvconst}
{We make a few remarks here about Theorem~\ref{thm:NCV}. (i) Note that $\EE \|(\widetilde \vv^{(k_0-1)})^{-\frac{1}{2}}\|_1\le \frac{n}{c}$. Hence, we have the ergodic sublinear convergence $O(\frac{1}{\sqrt{K-k_0+1}})$; (ii) The existence of $c>0$ such that $\widetilde v_i^{(k_0-1)}\ge c^2, \forall i\in [n]$ almost surely is a mild assumption. If $k_0$ is large, then it is likely that $\widetilde v_i^{(k_0-1)}\ge \widehat v_i^{(k_0-1)}\gtrapprox (1-\beta_2) G_{i,\infty}^2$, where $G_{i,\infty}$ is an almost-sure bound on $|\nabla_i f(\vx;\xi)|$; (iii) Suppose $K\ge 2$ and $k_0=\lceil \frac{K}{2}\rceil$. If $\alpha = O(1)$ and $\tau=o(K^{\frac{1}{4}})$, then $\frac{C_2}{c}(\sqrt{C_1}+\frac{C_2}{c}) \ll C_1$ when $K$ is large. Hence, we observe from \eqref{eq:ncvconstant2} that in this case, the delay will just slightly affect the convergence speed, and we can achieve nearly-linear speed-up.} 
\end{remark}

\section{Numerical experiments}\label{sec:numerical}
In this section, we conduct numerical experiments on the proposed algorithm APAM under both shared-memory and distributed settings.  
We compare APAM to the non-parallel and sync-parallel versions of AMSGrad, and also to the async-parallel nonadaptive SGM. Notice that AMSGrad has been shown in the literature to converge significantly faster than a non-adaptive SGM or a momentum SGM, and in addition, it performs similarly well as Adam, another popularly used adaptive SGM not guaranteed to converge. 
Both async- and sync-parallel methods are implemented in C++, using openMP for shared-memory parallelization and using MPI for distributed communication. 
They are also implemented in Python for tests with large-scale datasets, using  MPI4PY for distributed communication. Tests in subsections \ref{subsection:5.1}--\ref{subsection:5.3} are run with the C++ implementation on a Dell workstation with 32 CPU cores, 64 GB memory, and two Quadro RTX 5000 GPUs.
Tests in subsections \ref{subsection:5.4}--\ref{subsection:5.5}  are run with  the Python implementation on the same workstation.


In our tests, we use five datasets: rcv1 from LIBSVM \cite{chang2011libsvm}, MNIST \cite{lecun1998gradient},  Cifar10 \cite{krizhevsky2009learning}, CINIC10 \cite{darlow2018cinic}, and Imagenet32$\times$32 \cite{chrabaszcz2017downsampled}. Their characteristics are listed in Table~\ref{table:data}.

\begin{table}[ht]
\caption{Characteristics of the tested datasets.\vspace{-0.2cm}}\label{table:data}
\begin{center}
{\small
\begin{tabular}{ccccc}
\hline
name & train samples & test samples & features & classes \\\hline
rcv1 & 20,242 & 677,399 & 47,236 & 2\\
MNIST & 60,000 & 10,000 & 28$\times$28 & 10 \\
Cifar10 & 50,000 & 10,000  & 32$\times$32$\times$3  & 10 \\
CINIC10 & 180,000 & 90,000 &  32$\times$32$\times$3 & 10\\
Imagenet32$\times$32 & 1,281,167  & 50,000 & 32$\times$32$\times$3 & 1000 \\
\hline
\end{tabular}
}
\end{center}
\end{table}

\subsection{Performance of APAM}\label{subsection:5.1}
In this subsection, we demonstrate the convergence behavior and parallelization speed-up of APAM on solving both convex and non-convex problems. We apply APAM to solve the logistic regression (LR) problem and to train a 2-layer fully-connected neural network (NN). The LR problem is convex while the neural network training is non-convex. For the LR problem, we use the rcv1 data, and for the 2-layer NN, we use the MNIST data. We set the number of neurons in the hidden layer to $50$ in the 2-layer NN, and we use the hyperbolic tangent function as the activation. 
The initial iterates for both problems are set as the standard Gaussian. While computing a stochastic gradient, the mini-batch size is set to 64 and 32 respectively for the two problems. We set the learning rate to $\alpha_k=10^{-2},\forall\,k$ for the LR problem and $\alpha_k=5\times 10^{-4},\forall\, k$ for the 2-layer NN training. The weight parameters are set to $\beta_1=0.9$ and $\beta_2=0.999$ in this test and also all the other tests.

For both problems, we report the wall-clock time and prediction accuracy on the testing data. In addition, we report the objective error, i.e., the distance of the objective value to the optimal value, for the convex LR problem and the training accuracy for the non-convex 2-layer NN problem. The results for the LR problem are shown in Fig.~\ref{fig:APAM_rcv1} and those for the 2-layer NN training in Fig.~\ref{fig:APAM_DNN2_mnist}. From the results, we see that the convergence speed of APAM, measured by the objective error or prediction accuracy versus epoch number, keeps almost the same when the number of threads used for the shared-memory parallel computing changes. This observation indicates that the convergence speed of APAM is just slightly affected by the information delay. For both problems, over 16x speed-up is achieved in terms of the running time while 32 threads are used. 

\begin{figure}[ht]
\caption{running time, objective errors and testing accuracy by APAM with openMP for logistic regression on rcv1 dataset.\vspace{-0.1cm}}
    \label{fig:APAM_rcv1}
    \centering
{\small\begin{tabular}{|c|cccccc|}
\hline
\#thread & 1 & 2 & 4 & 8 & 16 & 32\\\hline
time (sec) & 321.0 & 168.1 & 88.3 & 48.3 & 27.6 & 19.7 \\\hline
\end{tabular}
}\\
    \includegraphics[width=0.3\textwidth]{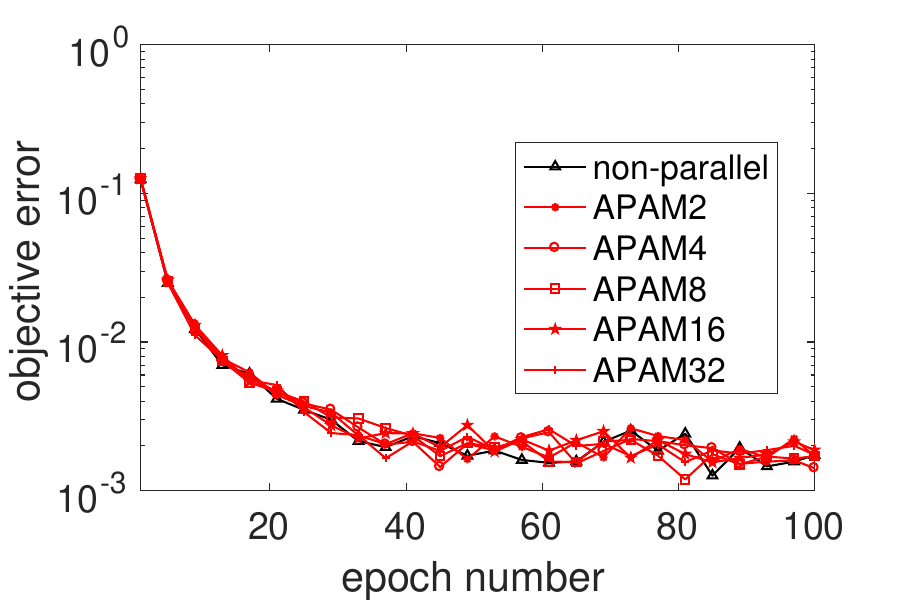}
    \includegraphics[width=0.3\textwidth]{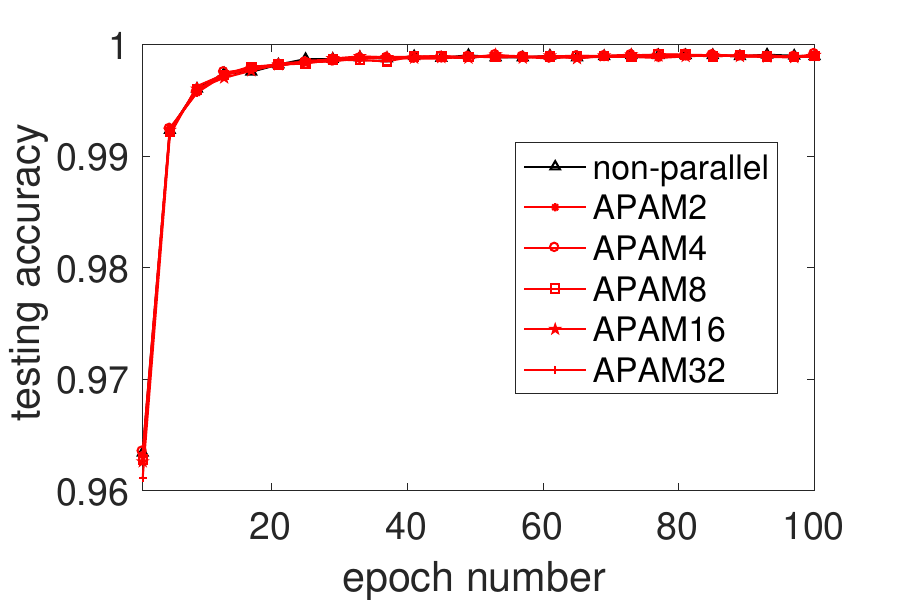}
\end{figure}

\begin{figure}[ht]
\caption{running time, training accuracy, and testing accuracy by APAM with openMP for learning a 2-layer fully-connected neural network on MNIST dataset.\vspace{-0.1cm}}
    \label{fig:APAM_DNN2_mnist}
    \centering
    {\small\begin{tabular}{|c|cccccc|}
\hline
\#thread & 1 & 2 & 4 & 8 & 16 & 32\\\hline
time (sec) & 2396 & 1259 & 658 & 361 & 201 & 134 \\\hline
\end{tabular}
}\\
    \includegraphics[width=0.3\textwidth]{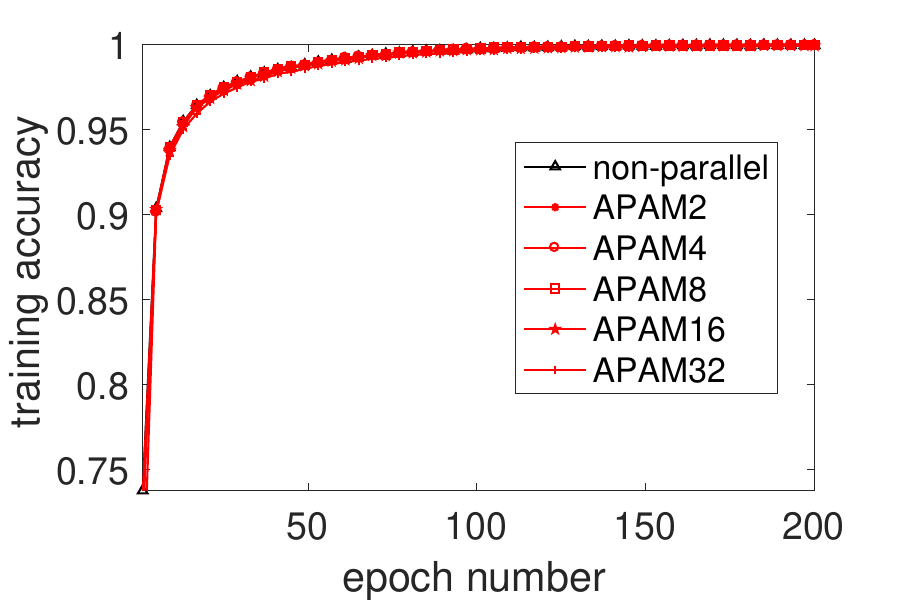}
    \includegraphics[width=0.3\textwidth]{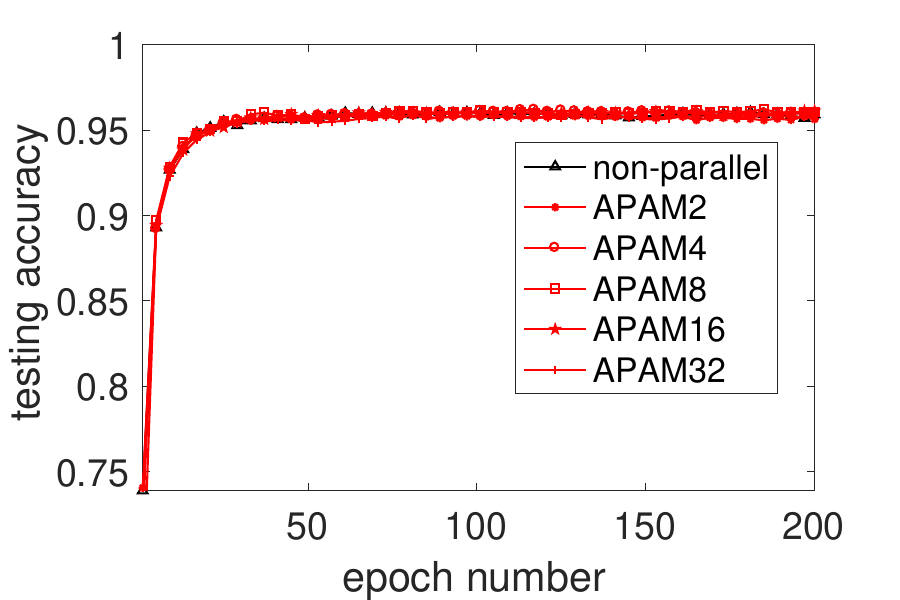}
\end{figure}

\subsection{Comparison to the nonadaptive SGM} 
In this subsection, we compare APAM to the async-parallel nonadaptive SGM. The latter method is implemented in a way similar to how we implement APAM but with a nonadaptive update. We test the two methods on training the 2-layer neural network in the previous subsection and the LeNet5 network \cite{lecun1998gradient} by using the MNIST dataset. LeNet5 has 2 convolutional,  2 max-pooling, and 3 fully-connected layers. For the 2-layer network, we use openMP shared-memory parallelism on the two methods. The mini-batch size in computing a stochastic gradient is set to 32 for both methods. The parameters of APAM are set the same as those in the previous subsection, and the learning rate of the nonadaptive SGM is tuned to $10^{-3}$ for the highest testing accuracy. For the LeNet5 network, we conduct distributed computing with MPI. The mini-batch size is set to 40 for both methods, and the learning rate is tuned to $10^{-4}$ and $10^{-3}$ respectively for APAM and the async-parallel nonadaptive SGM, for the highest testing accuracy.

For the openMP implementation, we compare the performance of the two methods by running them with 1, 8, or 32 threads. For the MPI implementation, we compare their performance with one master process and 1, 5, or 20 worker processes. When one thread or one worker is used, the methods become nonparallel. Both methods are run to 200 epochs. The results are shown in Fig.~\ref{fig:sgd_APAM_DNN2_mnist} for the openMP implementation and in Fig.~\ref{fig:sgd_APAM_LeNet_mnist} for the MPI implementation. Because the nonadaptive SGM update is cheaper than the APAM update, we plot the accuracy versus the running time. From the convergence curves, we see that the convergence speed of both APAM and the async-parallel nonadaptive SGM is slightly affected by the information delay. In addition, we see that APAM gives significantly higher training and testing accuracy than the nonadaptive SGM within the same amount of running time.

\begin{figure}[ht]
\caption{training and testing accuracy by APAM and the async-parallel nonadaptive SGM with openMP for a 2-layer fully-connected neural network on MNIST.\vspace{-0.2cm}}
    \label{fig:sgd_APAM_DNN2_mnist}
    \centering
    \includegraphics[width=0.3\textwidth]{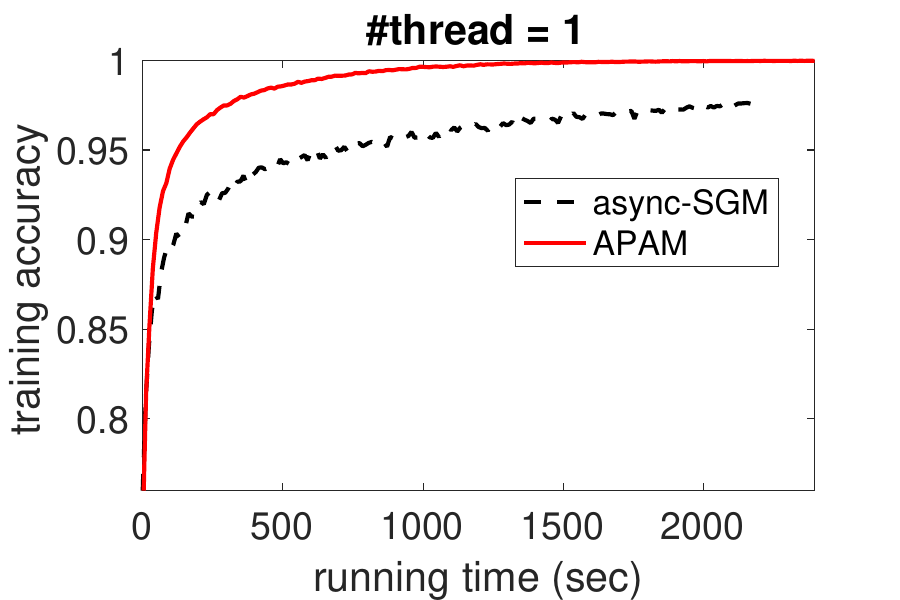}
    \includegraphics[width=0.3\textwidth]{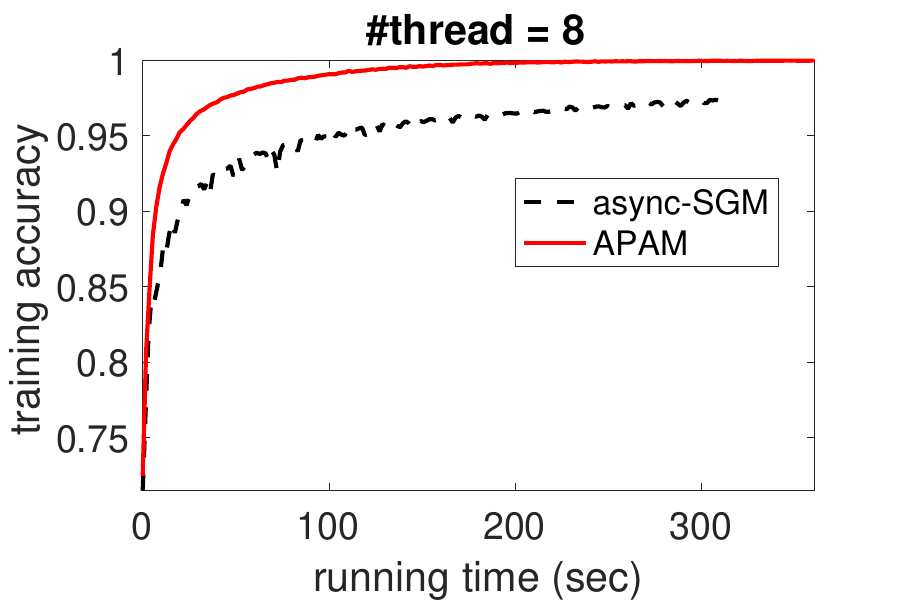}
    \includegraphics[width=0.3\textwidth]{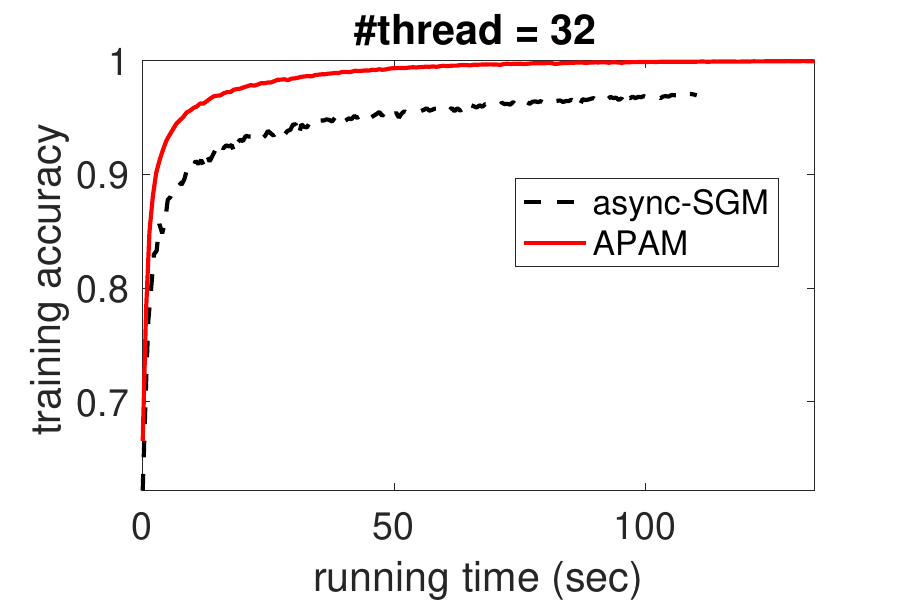}\\
    \includegraphics[width=0.3\textwidth]{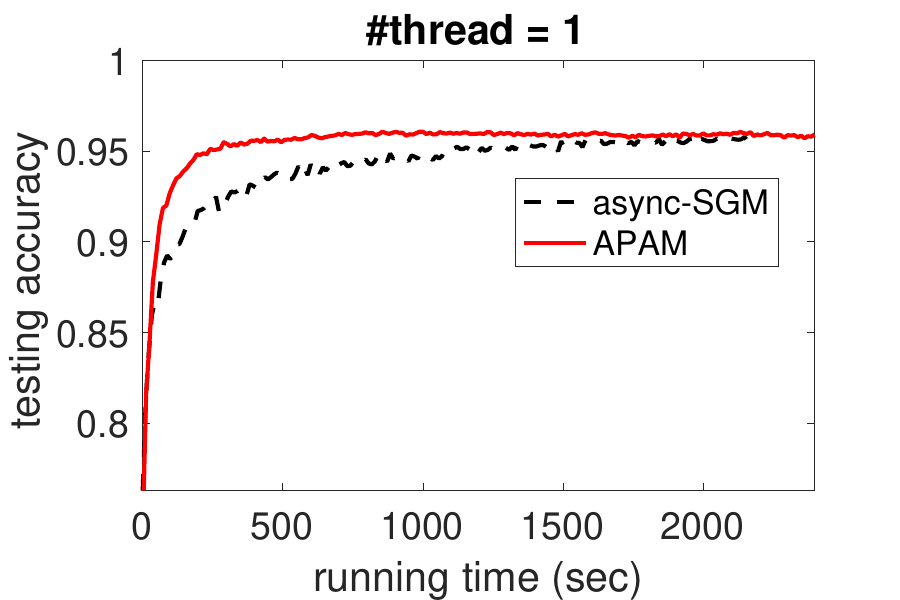}
    \includegraphics[width=0.3\textwidth]{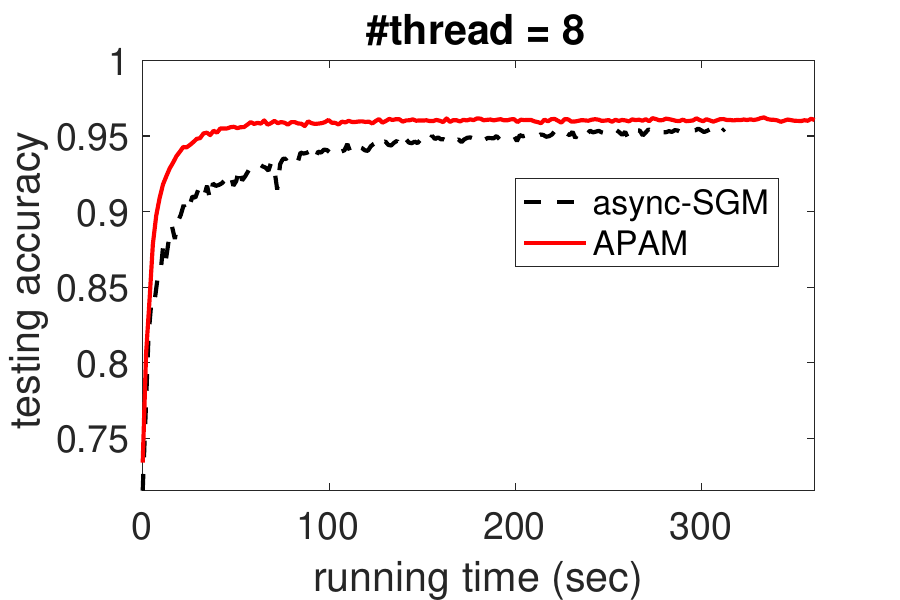}
    \includegraphics[width=0.3\textwidth]{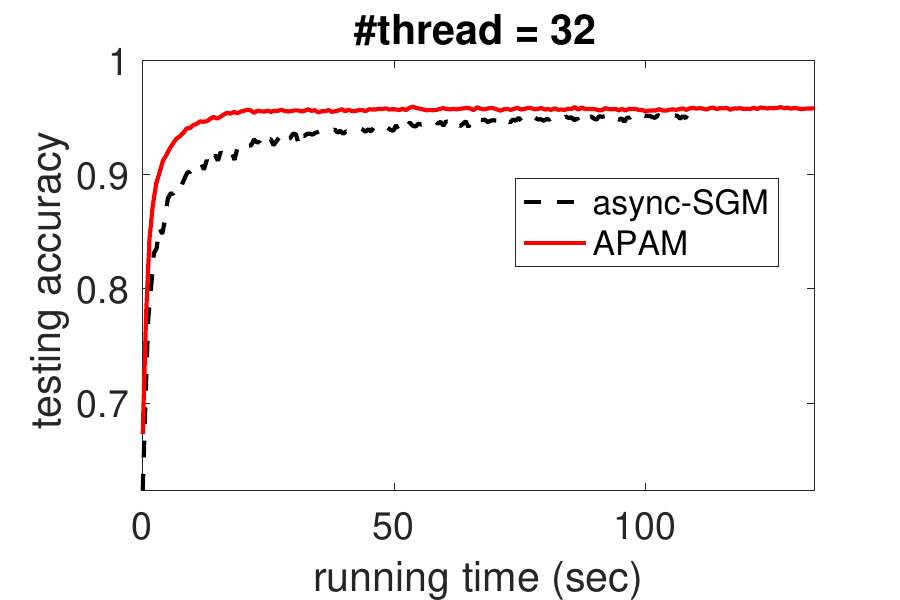}
\end{figure}

\begin{figure}[ht]
\caption{training and testing accuracy by APAM and the async-parallel nonadaptive SGM with MPI for the LeNet5 neural network on MNIST dataset.\vspace{-0.1cm}}
    \label{fig:sgd_APAM_LeNet_mnist}
    \centering
    \includegraphics[width=0.3\textwidth]{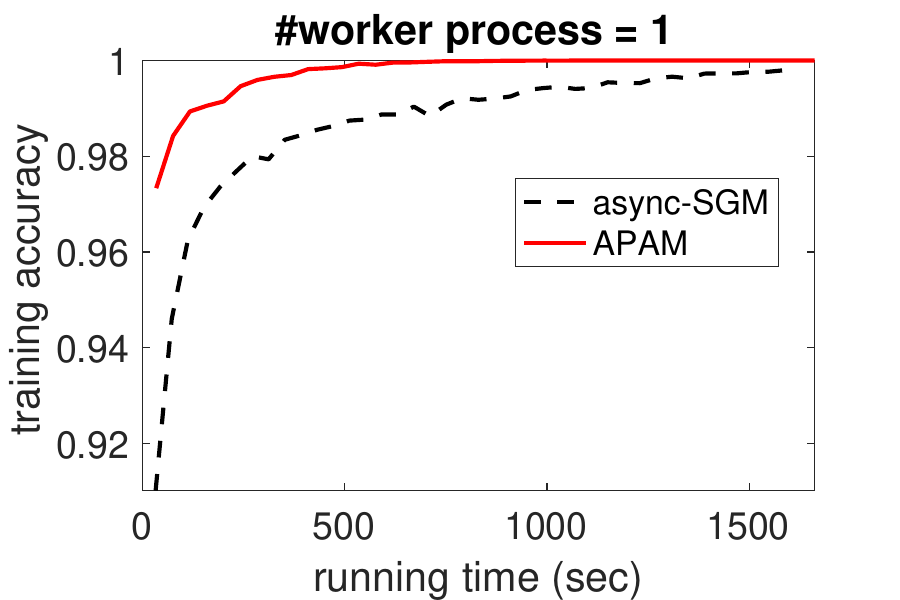}
    \includegraphics[width=0.3\textwidth]{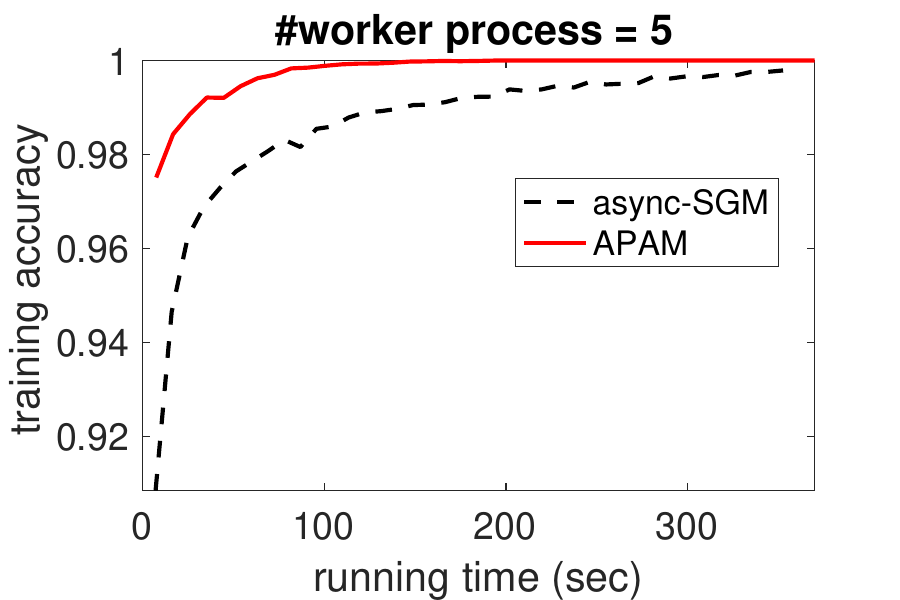}
\includegraphics[width=0.3\textwidth]{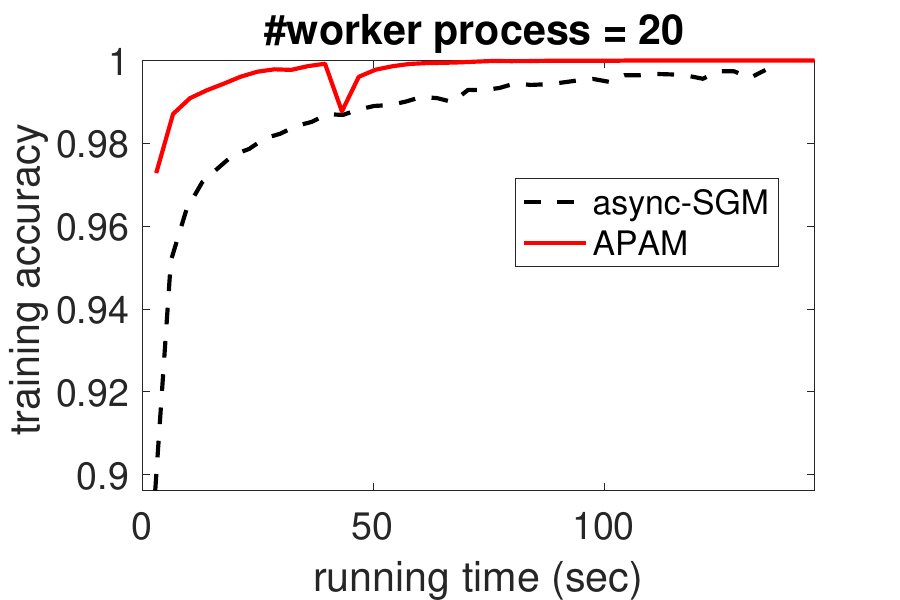}\\
    \includegraphics[width=0.3\textwidth]{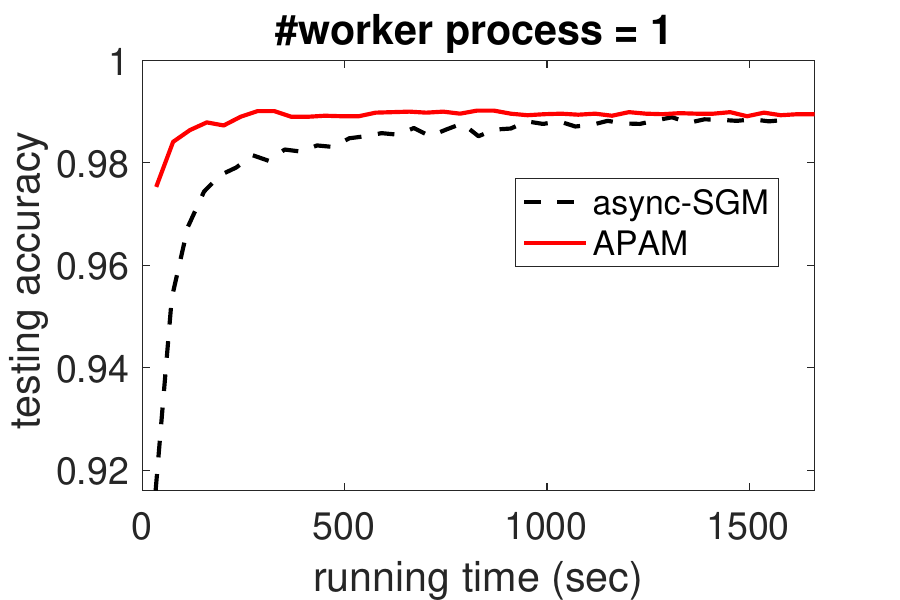}
        \includegraphics[width=0.3\textwidth]{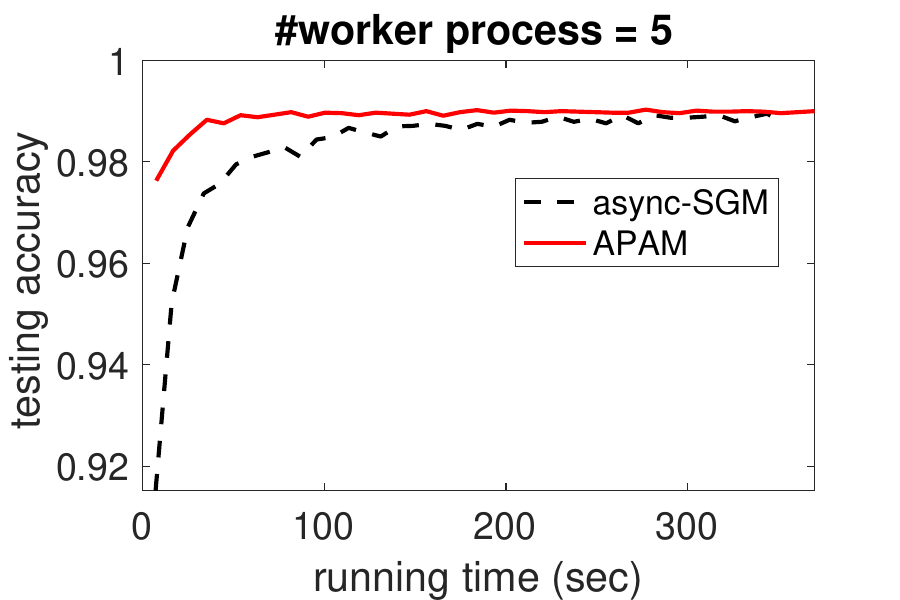}
        \includegraphics[width=0.3\textwidth]{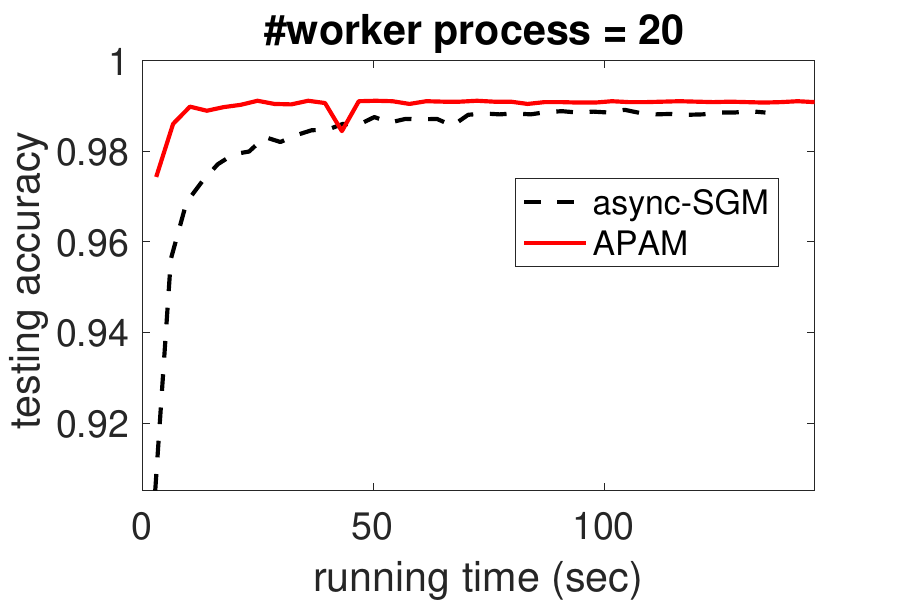}
\end{figure}

\subsection{Comparison to sync-parallel method}\label{subsection:5.3}
In this subsection, we compare APAM to its sync-parallel counterpart. We test them on training the 2-layer neural network used in the previous two subsections and on training the AllCNN network in \cite{springenberg2014striving} without data augmentation. The MNIST data is used for the 2-layer network and Cifar10 for the AllCNN network.  
AllCNN has 9 convolutional layers. The parameters of APAM and its sync-parallel counterpart are set to the same values. On training the 2-layer network, we adopt the same parameter settings as those in the previous two subsection. On the AllCNN, we conduct distributed computing with MPI. We set the mini-batch size to 40 and tune the learning rate to $\alpha_k=10^{-4},\forall\, k$ based on testing accuracy.

The running time for the 2-layer network is shown in Fig.~\ref{fig:comp-sync} and the results for the AllCNN in Fig.~\ref{tab:comp-sync-mpi}. Fig.~\ref{fig:APAM_DNN2_mnist} shows that on the 2-layer network, APAM gives almost the same accuracy curves while the number of threads used in the training changes. Hence, the results in Fig.~\ref{fig:APAM_DNN2_mnist} and Fig.~\ref{fig:comp-sync} indicate that APAM can achieve significantly higher parallelization speed-up than its sync-parallel counterpart to reach the same training/testing accuracy, especially when 16 or 32 threads are used. The sync-parallel method achieves lower paralleization speed-up by using 32 threads than that by using 8 or 16 threads. This is possibly because of memory congestion. For the AllCNN, we see that in the beginning, APAM produces lower accuracy as more worker processes are used, and this should be because the delay slows down the convergence speed. However, to reach the final highest accuracy, APAM with different number of worker processes takes almost the same number of epochs. Therefore, the results indicate that APAM again has significantly higher speed-up than its sync-parallel counterpart to reach the highest training/testing accuracy, especially when 10 or 20 worker processes are used.

\begin{figure}[ht]
    \caption{running time (hour) of APAM and the sync-parallel AMSGrad with openMP by different number of threads for training a 2-layer fully-connected network on MNIST dataset. The training/testing accuracy results are shown in Fig.~\ref{fig:APAM_DNN2_mnist}.\vspace{-0.1cm}}
    \label{fig:comp-sync}
    \centering
    \includegraphics[width=0.35\textwidth]{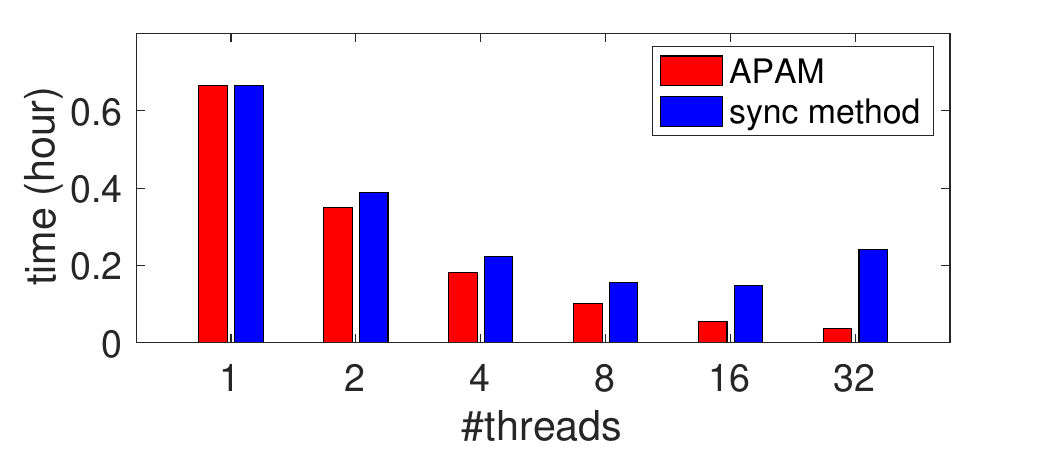}
\end{figure}

\begin{figure}[ht]
\caption{running time (hour) and prediction accuracy by APAM and the sync-parallel AMSGrad with MPI implementation for training the AllCNN network without data augmentation on Cifar10 dataset.\vspace{-0.1cm}}
    \label{tab:comp-sync-mpi}
    \centering
\includegraphics[width=0.35\textwidth]{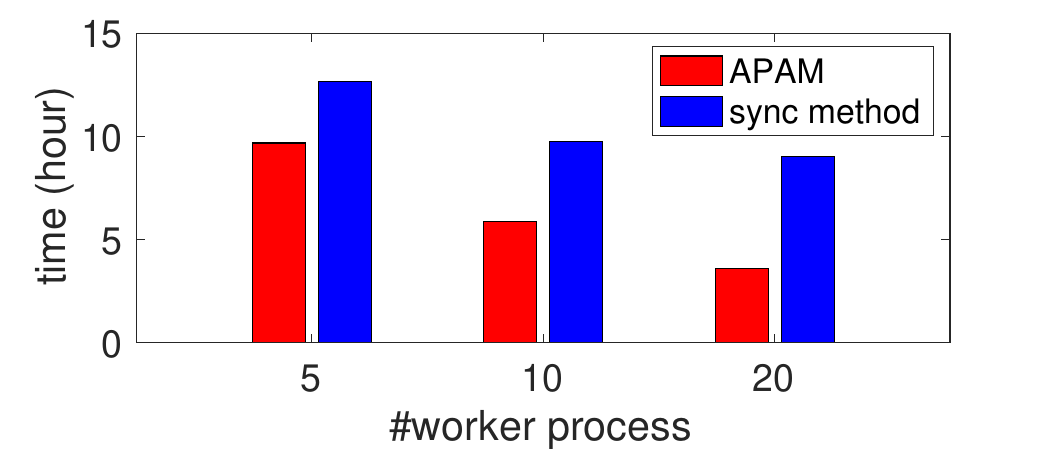}\\
\includegraphics[width=0.3\textwidth]{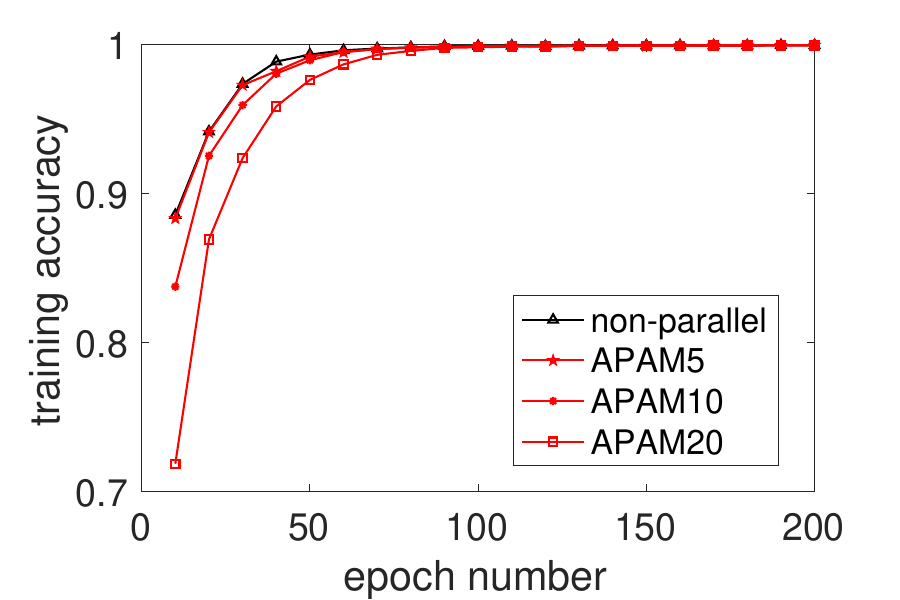}
\includegraphics[width=0.3\textwidth]{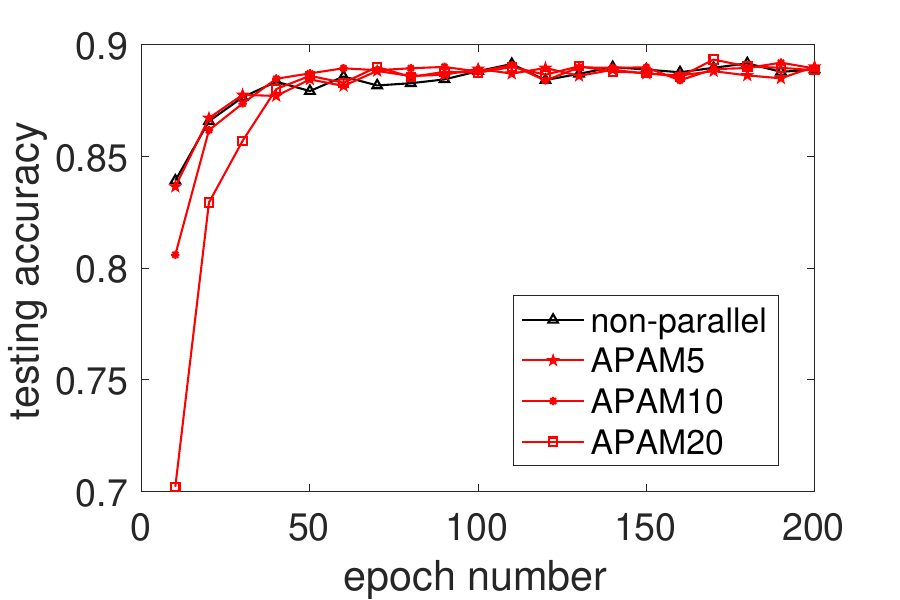}
\end{figure}

\subsection{Results with artificial delay}\label{subsection:5.4}
In this subsection, we test the effect of the delay in our algorithm APAM, by artificially injecting different delays like \cite{assran2020asynchronous}.
For a given maximum delay $\tau$, 
we artificially select the delay $\tau_k$ in iteration $k$ from $\{0,1,...,\min\{\tau,k\}\}$ uniformly at random, i.e., the stochastic gradient $\vg^{(k)}$ is evaluated at an iterate that is selected from $\{\vx^{(k)}, \vx^{(k-1)}, ..., \vx^{(k-\min\{\tau,k\})}\}$ uniformly at random.
We test APAM on training the AllCNN network on Cifar10 with different maximum delays and the same parameter settings as those in the previous subsection.

Fig.~\ref{fig:artificial_delay} plots the curves for different values of $\tau$. 
In all cases, APAM can converge to almost the same final highest accuracy. 
When $\tau \le 20$, APAM takes almost the same number of epochs to reach the final highest accuracy as the no-delay case. 
When $\tau \ge 50$, the negative effect of delay on the convergence becomes obvious. APAM with a larger maximum delay converges slower and needs more epochs to achieve the final highest accuracy. APAM with the maximum delay of 200 converges the slowest but it can still achieve the final highest accuracy after about 300 epochs.

\begin{figure}[htbp]
\caption{predication accuracy by APAM for training the AllCNN network on Cifar10 dataset with Python implementation and artificial delay. \vspace{-0.1cm}}
    \label{fig:artificial_delay}
    \centering 
\includegraphics[width=0.3\textwidth]{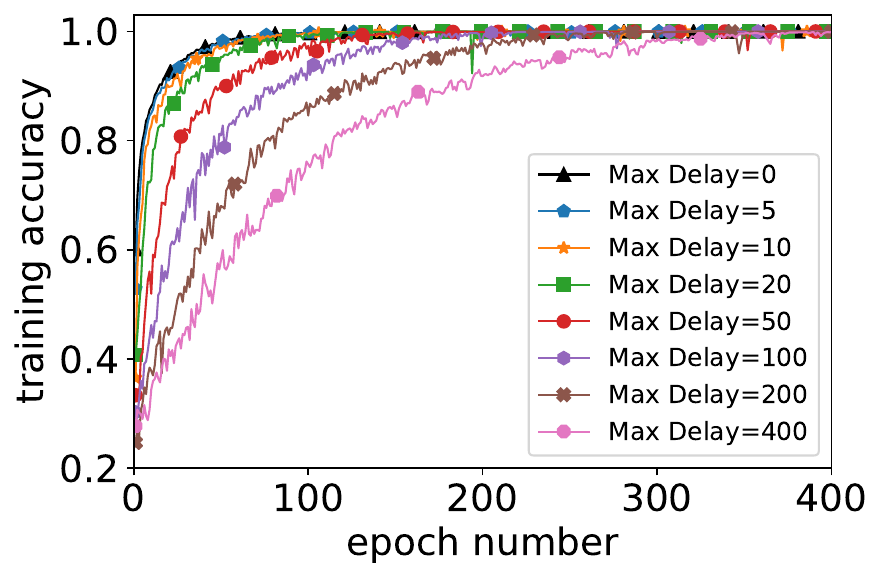}
\includegraphics[width=0.3\textwidth]{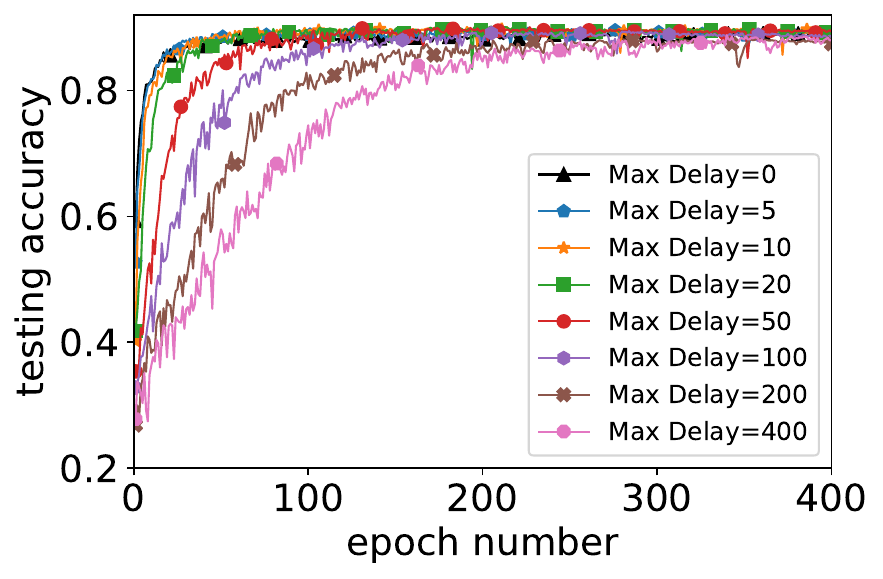}
\end{figure}
 
\subsection{Tests on larger datasets}\label{subsection:5.5}
In this subsection, we test APAM and its sync-parallel counterpart on larger neural networks and larger datasets.
We train two networks: Resnet18 \cite{he2016deep} that is a deep residual network with 18 convolutional layers, and WRN-28-5 \cite{zagoruyko2016wide} that is a wide residual network with 28 convolutional layers and whose widening factor is 5. 
 Resnet18 is used for classifying the CINIC10 data, with the mini-batch size set to 80 and the learning rate tuned to $10^{-4}$.
 WRN-28-5 is used for classifying the Imagenet32$\times$32 data, with the mini-batch size set to 100 and the learning rate tuned to $10^{-3}$. 

 The training is first run on CPUs to compare the time of APAM and its sync counterpart. 
 Because of the problem size, it takes very long time for one update, and we only run the training to one epoch. Fig~\ref{fig:cpu_time} shows the running time. 
From the figure, we see again that APAM has significantly higher speed-up than its sync-parallel counterpart. 
In more details, the running time for both trainings by APAM decreases as the number of workers increases, and it is reduced almost by a half as the workers increase from 5 to 10 and from 10 to 20. However, for the sync-parallel counterpart, the speed-up is only observed when the number of workers increases from 1 to 5, and as the number of workers further increases, it takes longer time for training Resnet18 on the CINIC10 data.

\begin{figure}[htbp]
\caption{running time (hour) on CPU by APAM and the sync-parallel AMSGrad with Python and MPI4PY implementation, for training the Resnet18 network on  CINIC10 dataset (Left) and WRN-28-5 on Imagenet32$\times$32 dataset (Right); both for one epoch.\vspace{-0.1cm}}
    \label{fig:cpu_time}
    \centering
\includegraphics[width=0.35\textwidth]{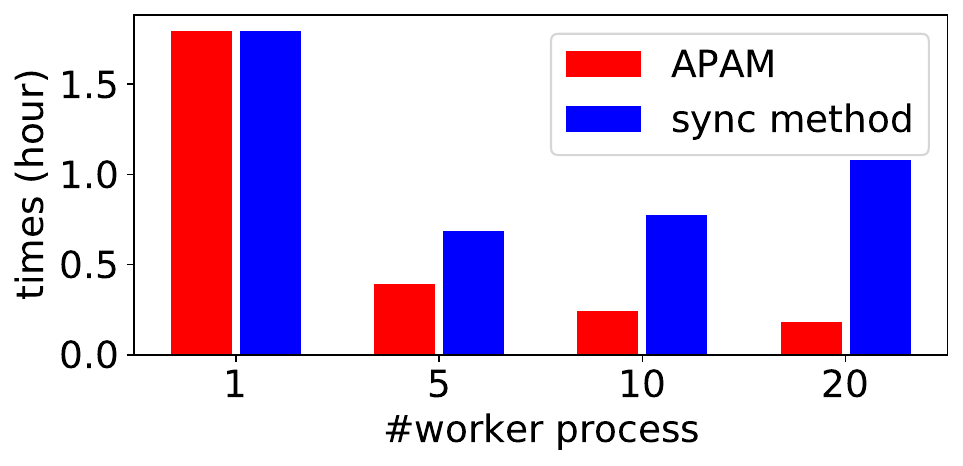}
\includegraphics[width=0.35\textwidth]{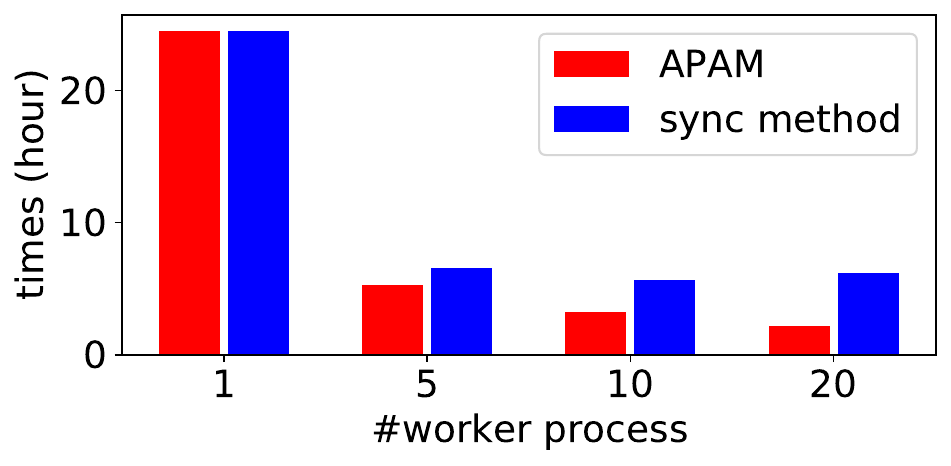}
\end{figure}

Then we run APAM with two GPUs to see how the delay affects the convergence. 
As there are only two GPUs, we assign $\lceil \frac{p}{2}\rceil$ workers on one GPU and $\lfloor \frac{p}{2}\rfloor$ workers on the other one, if $p$ workers are employed. 
Due to the memory limitation, $p$ is set up to 10. 
Fig.~\ref{fig:CINIC10} and Fig.~\ref{fig:Imagenet32} show the prediction accuracy of the training on Resnet18 and WRN-28-5 respectively. From the figures, we see that APAM produces lower accuracy for the beginning epochs as more workers are used. This indicates that the delay slows down the convergence speed. Nevertheless, APAM achieves almost the same final highest accuracy with different numbers of workers.  
It is worth to mention that to train the models on the large datasets for many epochs takes a long time, even on GPUs. The number of epochs is selected such that the training takes about one day by using both GPUs. This way, we could run to 200 epochs for training Resnet18 on CINIC10 and only 40 epochs for training WRN-28-5 on Imagenet32$\times$32.

\begin{figure}[htbp]
\caption{prediction accuracy by APAM and the sync-parallel AMSGrad with Python and MPI4PY implementation for training the Resnet18 network on CINIC10 dataset.\vspace{-0.1cm}}
    \label{fig:CINIC10}
    \centering
\includegraphics[width=0.3\textwidth]{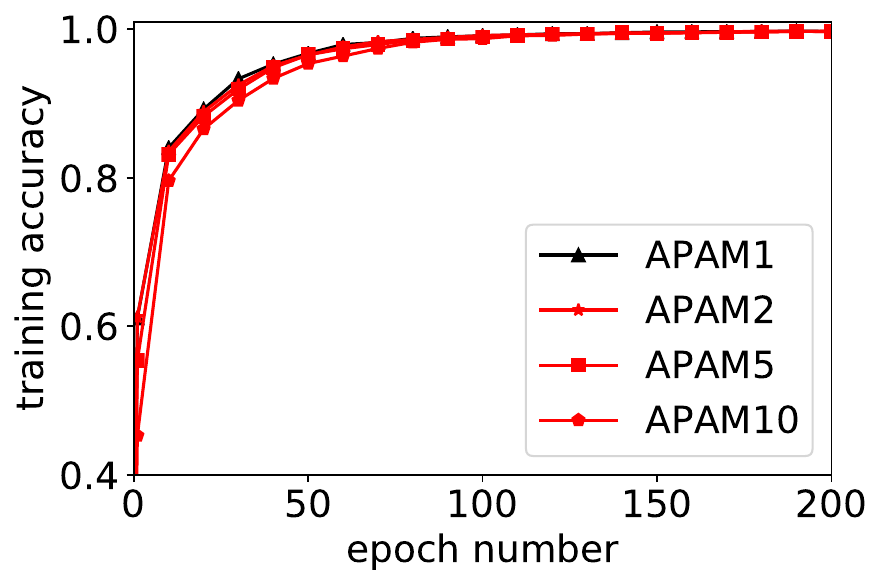}
\includegraphics[width=0.3\textwidth]{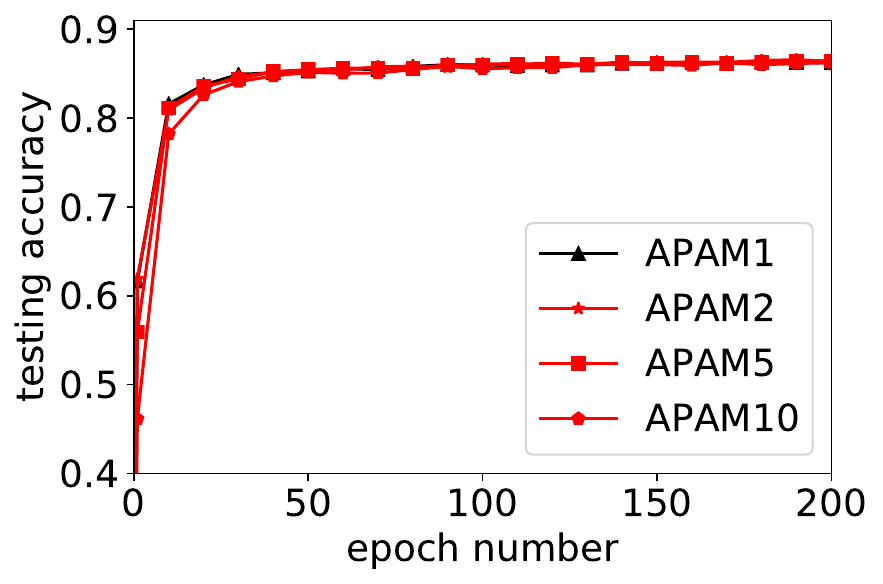}
\end{figure}

\begin{figure}[htbp]
\caption{prediction accuracy by APAM and the sync-parallel AMSGrad with Python and MPI4PY implementation for training the WRN-28-5 network on Imagenet32$\times$32 dataset. 
The first row is about the conventional accuracy that measures the proportion of images for which the predicted class (the one with the highest probability) matches the true class. The second row is about the top 5 accuracy that measures the proportion of images for which one of the five classes with top 5 highest probability matches the true class.
\vspace{-0.1cm}}
    \label{fig:Imagenet32}
    \centering 
\includegraphics[width=0.3\textwidth]{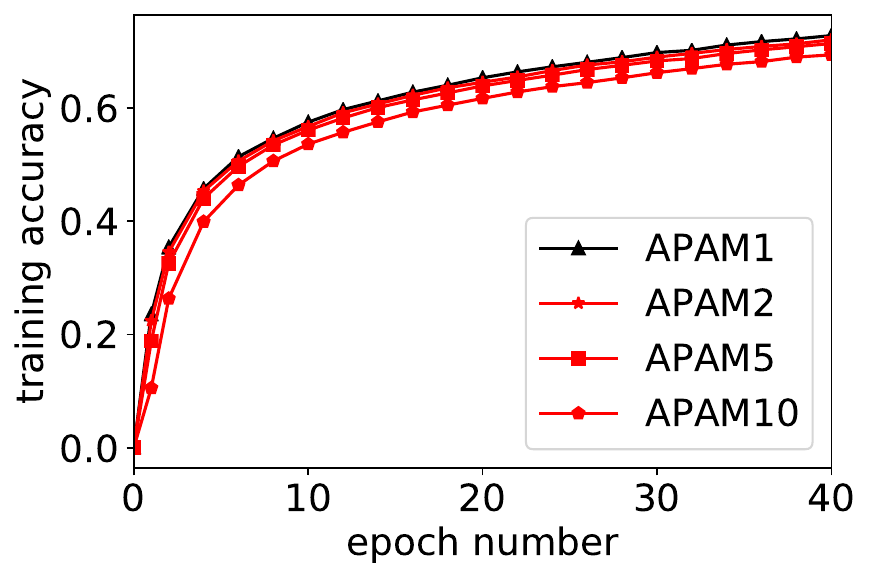}
\includegraphics[width=0.3\textwidth]{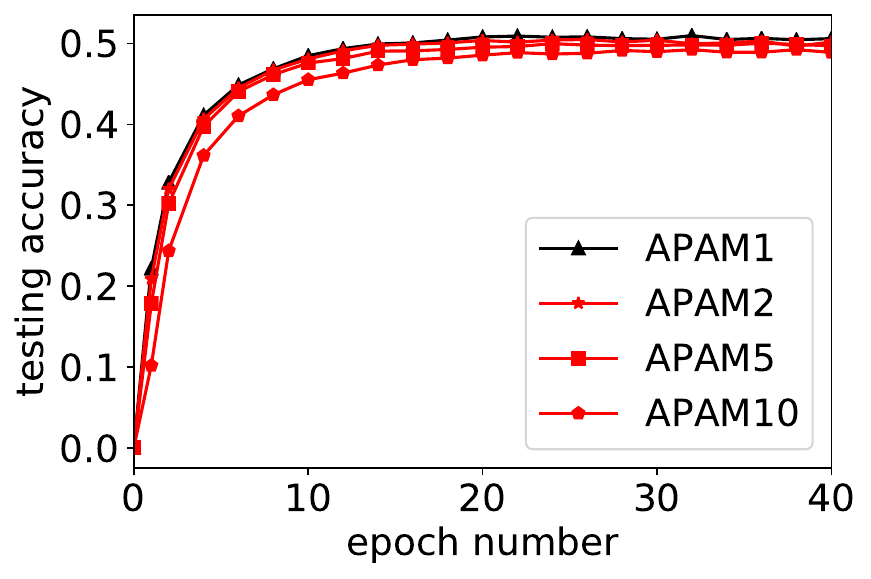}\\
\includegraphics[width=0.3\textwidth]{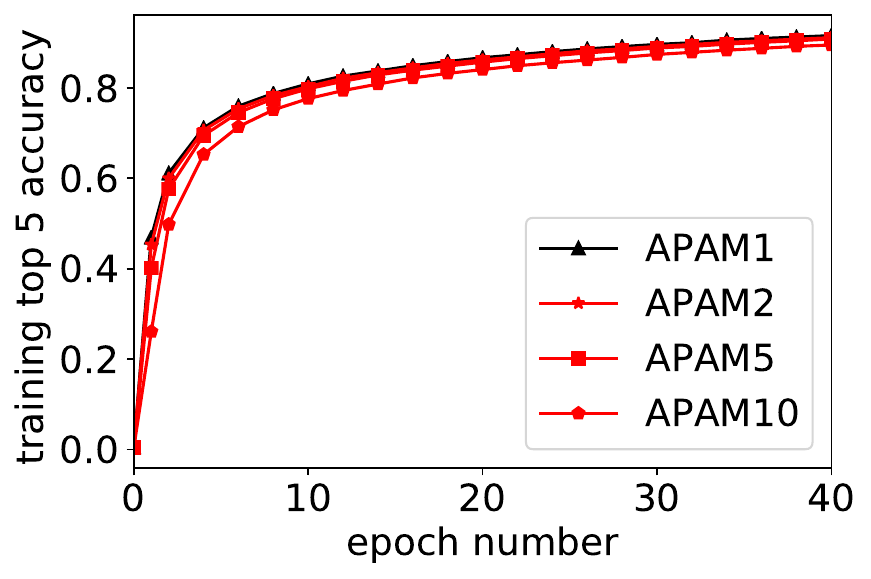}
\includegraphics[width=0.3\textwidth]{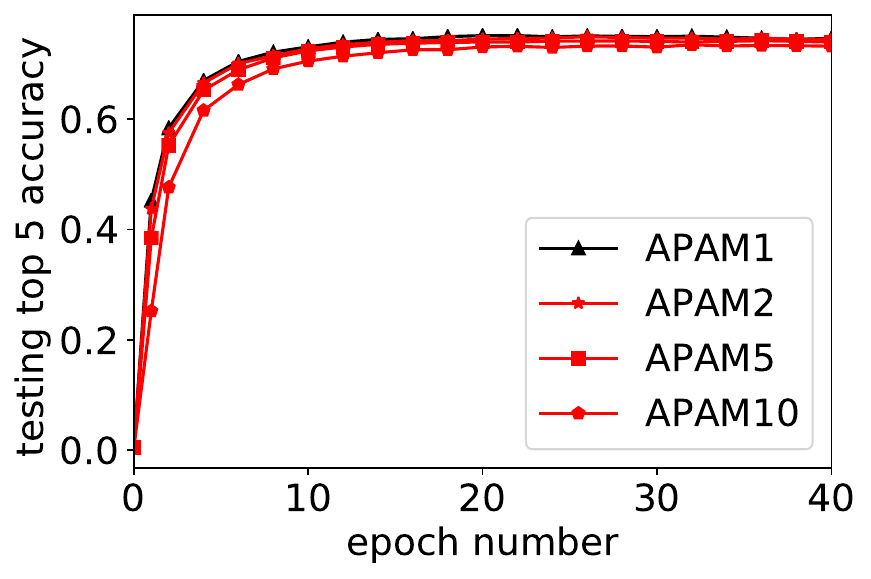}
\end{figure}

\section{Concluding remarks}\label{sec:conclusion}
We have presented an asynchronous parallel adaptive stochastic gradient method, named APAM, based on AMSGrad. Convergence rate results are established for both constrained convex and unconstrained non-convex cases. The results show that the delay has little effect on the convergence speed, if it is upper bounded by $\tau=o(K^{\frac{1}{4}})$, where $K$ is the maximum number of iterations. Numerical experiments on both convex and non-convex machine learning problems demonstrate significant advantages of the proposed method over its synchronous counterpart and also an asynchronous parallel nonadaptive method, in both shared-memory and distributed environment.

\begin{acknowledgements}
The authors would like to thank three anonymous reviewers for their valuable comments and suggestions to improve the quality of the paper and also for the careful testing on our codes. The work of Y. Xu, Y. Yan and C. Sutcher-Shepard is partly supported by NSF grant DMS-2053493 and the Rensselaer-IBM AI Research Collaboration, part of the IBM AI Horizons Network. J. Chen is supported in part by DOE Award DE-OE0000910.
\end{acknowledgements}


\bibliographystyle{abbrv}


\appendix

\section{Proofs of lemmas in section \ref{sec:cvx}}

\textbf{Proof of Lemma \ref{lem:ineq-from-opt}}~~~~
From the update of $\vx$ in \eqref{eq:update-x}, we have the optimality condition
$$\vzero\in \cN_X(\vx^{(k+1)}) + \sqrt{\widehat\vv^{(k)}}(\vx^{(k+1)}-\vx^{(k)}) + \alpha_k \vm^{(k)},$$
where $\cN_X(\vx)$ denotes the normal cone of $X$ at $\vx$. Hence, it follows
\begin{equation}\label{eq:optimality-cond}
\left\langle \vx^{(k+1)} - \vx, \sqrt{\widehat\vv^{(k)}}(\vx^{(k+1)}-\vx^{(k)}) + \alpha_k \vm^{(k)} \right\rangle \le 0, \, \forall\, \vx \in X.
\end{equation}

By the update of $\vm$ in \eqref{eq:update-m}, it holds
\begin{align*}
&~\left\langle \vx^{(k+1)} - \vx, \vm^{(k)} \right\rangle\cr
= &~ \left\langle \vx^{(k+1)} - \vx^{(k)}, \vm^{(k)} \right\rangle + \left\langle \vx^{(k)} - \vx, \vm^{(k)} \right\rangle\cr
= & ~ \left\langle \vx^{(k+1)} - \vx^{(k)}, \vm^{(k)} \right\rangle + (1-\beta_1) \left\langle \vx^{(k)} - \vx, \vg^{(k)} \right\rangle + \beta_1 \left\langle \vx^{(k)} - \vx, \vm^{(k-1)} \right\rangle.
\end{align*}
Recursively using the above relation, we have
\begin{equation}\label{eq:expand-m-term}
\left\langle \vx^{(k+1)} - \vx, \vm^{(k)} \right\rangle = \sum_{j=1}^k\beta_1^{k-j}\left(\langle \vx^{(j+1)}-\vx^{(j)}, \vm^{(j)}\rangle + (1-\beta_1)\langle \vx^{(j)} - \vx, \vg^{(j)} \rangle\right).
\end{equation}
In addition, it holds
\begin{align*}
&~\left\langle \vx^{(k+1)} - \vx, \sqrt{\widehat\vv^{(k)}}(\vx^{(k+1)}-\vx^{(k)})\right\rangle\\
=&~\frac{1}{2}\left(\|\vx^{(k+1)} - \vx\|_{\sqrt{\widehat\vv^{(k)}}}^2-\|\vx^{(k)} - \vx\|_{\sqrt{\widehat\vv^{(k)}}}^2+\|\vx^{(k+1)} - \vx^{(k)}\|_{\sqrt{\widehat\vv^{(k)}}}^2\right).
\end{align*}
Substituting the above two equations into \eqref{eq:optimality-cond} gives
\begin{align}\label{eq:expand-m-term-2}
&~\alpha_k \sum_{j=1}^k\beta_1^{k-j}\left(\langle \vx^{(j+1)}-\vx^{(j)}, \vm^{(j)}\rangle + (1-\beta_1)\langle \vx^{(j)} - \vx, \vg^{(j)} \rangle\right) \cr
\le&~- \frac{1}{2}\left(\|\vx^{(k+1)} - \vx\|_{\sqrt{\widehat\vv^{(k)}}}^2-\|\vx^{(k)} - \vx\|_{\sqrt{\widehat\vv^{(k)}}}^2+\|\vx^{(k+1)} - \vx^{(k)}\|_{\sqrt{\widehat\vv^{(k)}}}^2\right).
\end{align}

By the Young's inequality, we have
\begin{align*}
&~\sum_{k=1}^t\alpha_k \sum_{j=1}^k\beta_1^{k-j} \big\langle \vx^{(j+1)}-\vx^{(j)}, \vm^{(j)}\big\rangle \cr
= &~ \sum_{j=1}^t\sum_{k=j}^t \alpha_k \beta_1^{k-j} \big\langle \vx^{(j+1)}-\vx^{(j)}, \vm^{(j)}\big\rangle \cr
\ge &~ \sum_{j=1}^t\left(\sum_{k=j}^t \alpha_k \beta_1^{k-j}\right)\left(-\frac{\|\vx^{(j+1)}-\vx^{(j)}\|_{(\widehat\vv^{(j)})^{\frac{1}{2}}}^2}{2 \sum_{k=j}^t \alpha_k \beta_1^{k-j}}- \frac{\sum_{k=j}^t \alpha_k \beta_1^{k-j}}{2}\|\vm^{(j)}\|_{(\widehat\vv^{(j)})^{-\frac{1}{2}}}^2\right).
\end{align*}
Since $\sum_{k=j}^t \alpha_k \beta_1^{k-j}\le \frac{\alpha_j}{1-\beta_1}$, the above inequality implies
\begin{align}
&\sum_{k=1}^t\alpha_k \sum_{j=1}^k\beta_1^{k-j} \big\langle \vx^{(j+1)}-\vx^{(j)}, \vm^{(j)}\big\rangle \nonumber\\
\ge& - \sum_{j=1}^t \left(\frac{\|\vx^{(j+1)}-\vx^{(j)}\|_{(\widehat\vv^{(j)})^{\frac{1}{2}}}^2}{2 } + \frac{\alpha_j^2}{2(1-\beta_1)^2}\|\vm^{(j)}\|_{(\widehat\vv^{(j)})^{-\frac{1}{2}}}^2\right).\label{eq:bd-1st-term}
\end{align}
In addition, noting $\widehat\vv^{(k)}\ge \widehat\vv^{(k-1)}$ for all $k$, we have
\begin{align}\label{eq:bd-right-term}
&~-\sum_{k=1}^t\left(\|\vx^{(k+1)} - \vx\|_{\sqrt{\widehat\vv^{(k)}}}^2-\|\vx^{(k)} - \vx\|_{\sqrt{\widehat\vv^{(k)}}}^2\right)\cr
= &~ \left(-\|\vx^{(t+1)} - \vx\|_{\sqrt{\widehat\vv^{(t)}}}^2 + \sum_{k=2}^t\|\vx^{(k)} - \vx\|_{\sqrt{\widehat\vv^{(k)}}-\sqrt{\widehat\vv^{(k-1)}}}^2
+\|\vx^{(1)} - \vx\|_{\sqrt{\widehat\vv^{(1)}}}^2\right)\cr
\le &~ D_\infty^2 \left(\sum_{k=2}^t\|\sqrt{\widehat\vv^{(k)}}-\sqrt{\widehat\vv^{(k-1)}}\|_1+ \|\sqrt{\widehat\vv^{(1)}}\|_1\right) = D_\infty^2 \|\sqrt{\widehat\vv^{(t)}}\|_1.
\end{align}
Now summing \eqref{eq:expand-m-term-2} over $k=1$ to $t$, and using \eqref{eq:bd-1st-term} and \eqref{eq:bd-right-term}, we obtain the desired result.

\vspace{0.2cm}

\noindent\textbf{Proof of Lemma \ref{lem:bd-m-v}}~~~~
For each $i\in [n]$, let $G_i^{(k)}=\max_{j\le k}|g_i^{(j)}|$, and $\vG^{(k)}$ be the vector with the $i$-th component $G_i^{(k)}$. Note that for each $k\ge1$ and each $i\in [n]$, we have $\widehat v_i^{(k)}=\max\{\widehat v_i^{(k-1)},\, v_i^{(k)}\}=\max_{j\le k}v_i^{(j)}$, and in addition, $v_i^{(j)}=\sum_{\jmath=1}^j(1-\beta_2)\beta_2^{j-\jmath}(g_i^{(\jmath)})^2$. Hence,
\begin{equation}\label{eq:rel-hat-vik}
\widehat v_i^{(k)} = \max_{j\le k} \sum_{\jmath=1}^j(1-\beta_2)\beta_2^{j-\jmath}(g_i^{(\jmath)})^2,
\end{equation} 
and thus $\widehat v^{(k)}_i\ge (1-\beta_2) (G_i^{(k)})^2$. Therefore, noticing
\begin{equation}\label{eq:vm-k}
\vm^{(k)}=\sum_{j=1}^k(1-\beta_1)\beta_1^{k-j}\vg^{(j)},
\end{equation}
we have
\begin{align*}
 \|\vm^{(k)}\|_{(\widehat\vv^{(k)})^{-\frac{1}{2}}} = \|\vm^{(k)}\oslash (\widehat\vv^{(k)})^{\frac{1}{4}}\|
 \le& \frac{1}{(1-\beta_2)^{\frac{1}{4}}}\|\vm^{(k)}\oslash \sqrt{\vG^{(k)}}\|  \\
 \le& \frac{1}{(1-\beta_2)^{\frac{1}{4}}}\sum_{j=1}^k(1-\beta_1)\beta_1^{k-j} \big\|\vg^{(j)}\oslash \sqrt{\vG^{(k)}}\big\|, 
\end{align*}
and thus by the Cauchy-Schwarz inequality, it holds
\begin{align*}
    \|\vm^{(k)}\|_{(\widehat\vv^{(k)})^{-\frac{1}{2}}}^2
    \le& \frac{(1-\beta_1)^2}{(1-\beta_2)^{\frac{1}{2}}}\left(\sum_{j=1}^k\beta_1^{k-j}\right)\left(\sum_{j=1}^k\beta_1^{k-j} \big\|\vg^{(j)}\oslash \sqrt{\vG^{(k)}}\big\|^2\right)\\
    \le& \frac{1-\beta_1}{(1-\beta_2)^{\frac{1}{2}}}\sum_{j=1}^k\beta_1^{k-j} \big\|\vg^{(j)}\oslash \sqrt{\vG^{(k)}}\big\|^2.
\end{align*}
Now note that
$$\big\|\vg^{(j)}\oslash \sqrt{\vG^{(k)}}\big\|^2 = \sum_{i=1}^n\frac{|g_i^{(j)}|^2}{G_i^{(k)}}\le \sum_{i=1}^n |g_i^{(j)}| = \|\vg^{(j)}\|_1.$$
Together from the above two inequalities and Assumption~\ref{assump:bound-grad}, it follows that
$$\EE \|\vm^{(k)}\|_{(\widehat\vv^{(k)})^{-\frac{1}{2}}}^2 \le \frac{1-\beta_1}{(1-\beta_2)^{\frac{1}{2}}}\sum_{j=1}^k\beta_1^{k-j}\EE\|\vg^{(j)}\|_1 \le \frac{1-\beta_1}{(1-\beta_2)^{\frac{1}{2}}}\sum_{j=1}^k\beta_1^{k-j}G_1,$$
which implies the result in \eqref{eq:bd-m-v}.

\vspace{0.2cm}

\noindent\textbf{Proof of Lemma \ref{lem:nonexpansive}}~~~~Since $X$ is separable and $x_i^{(k+1)}=x_i^{(k)}$ if $\widehat v_i^{(k)}=0$, we have
$\vx^{(k+1)}=\Proj_{X}\left(\vx^{(k)}-\alpha_{k} \vm^{(k)}\oslash\sqrt{\widehat\vv^{(k)}} \right).$
In addition, notice that $\vx^{(k)} = \Proj_{X}\left(\vx^{(k)}\right)$, 
and thus the desired result follows from the non-expansiveness of the projection onto a convex set. 

\vspace{0.2cm}
\noindent\textbf{Proof of Lemma \ref{lem:bd-mg-v}}~~~~For each $i\in [n]$, let $G_i^{(k)}=\max_{j\le k}|g_i^{(j)}|$, and $\vG^{(k)}$ be the vector with the $i$-th component $G_i^{(k)}$. Then it follows from \eqref{eq:rel-hat-vik} that $\widehat v^{(k)}_i\ge (1-\beta_2) (G_i^{(k)})^2$. Hence, for $j\le k$,
$$\|\vg^{(j)}\oslash\sqrt{\widehat\vv^{(k)}}\|^2 \le \frac{1}{1-\beta_2}\sum_{i=1}^n \frac{(g_i^{(j)})^2}{(G_i^{(k)})^2}\le \frac{\|\vg^{(j)}\|_0}{1-\beta_2},$$
which gives \eqref{eq:bd-m-divide-v-1}.  
Furthermore, by \eqref{eq:vm-k}, it holds
$$\big\|\vm^{(k)}\oslash\sqrt{\widehat\vv^{(k)}}\big\|\le \sum_{j=1}^k(1-\beta_1)\beta_1^{k-j} \big\|\vg^{(j)}\oslash\sqrt{\widehat\vv^{(k)}}\big\|
\le \sum_{j=1}^k(1-\beta_1)\beta_1^{k-j} \frac{\sqrt{\|\vg^{(j)}\|_0}}{\sqrt{1-\beta_2}}, 
$$
which proves \eqref{eq:bd-m-divide-v-2}. The above inequality together with the Cauchy-Schwarz inequality implies \eqref{eq:bd-m-divide-v-3}. Hence, we complete the proof.

\section{Proofs of lemmas in section \ref{sec:ncvx}}

\textbf{Proof of Lemma \ref{lem:indbnd}}~~~~From \eqref{eq:vm-k}, we have $m^{(k)}_i=(1-\beta_1)\sum_{j=1}^k\beta_1^{k-j} g^{(j)}_i$, and thus applying triangle inequality and using the definition of $\vGam$ in \eqref{eq:fgbound} lead to $| m^{(k)}_i|\leq (1-\beta_1^k)\Gamma_i\le \Gamma_i$. A similar argument gives $\widehat v^{(k)}_i \leq  \Gamma_i^2.$ 
When Assumption~\ref{assump:bound-grad2} holds, we know that $\|\vg^{(k)}\|_\infty\le G_\infty$ almost surely, and that $\|\nabla F(\vx^{(k)})\|_\infty\le G_\infty,$ for all $k\in[K],$  which leads to the second part of this lemma.

\vspace{0.2cm}
\noindent\textbf{Proof of Lemma \ref{lem:zz}}~~~~From \eqref{defzk}, we have that for $k\geq 1$,
\begin{align*}
\vz^{(k+1)}-\vz^{(k)}&= \frac{1}{1-\beta_1}(\vx^{(k+1)}-\vx^{(k)})-\frac{\beta_1}{1-\beta_1}(\vx^{(k)}-\vx^{(k-1)}) \nonumber\\
&= -\frac{1}{1-\beta_1}\alpha_k(\widetilde{\vV}^{(k)})^{-\frac{1}{2}}\vm^{(k)} +\frac{\beta_1}{1-\beta_1} \alpha_{k-1}(\widetilde{\vV}^{(k-1)})^{-\frac{1}{2}}\vm^{(k-1)},
\end{align*}
where in the second equation, we have used \eqref{eq:update-x-unc} and Remark~\ref{rm:vtilde}. With $k=1$, the above equation gives \eqref{zzatone} by utilizing the update of $\vm^{(1)}$. Furthermore, it gives, by plugging the update of $\vm^{(k)}$,
\begin{align*}
&~\vz^{(k+1)}-\vz^{(k)}\\
=&~\frac{-1}{1-\beta_1}\alpha_k(\widetilde{\vV}^{(k)})^{-\frac{1}{2}}(\beta_1\vm^{(k-1)}+(1-\beta_1)\vg^{(k)}) +\frac{\beta_1}{1-\beta_1} \alpha_{k-1}(\widetilde{\vV}^{(k-1)})^{-\frac{1}{2}}\vm^{(k-1)}  \\
=&~\frac{\beta_1}{1-\beta_1}\left[\alpha_{k-1}(\widetilde\vV^{(k-1)})^{-\frac{1}{2}}-\alpha_k(\widetilde\vV^{(k)})^{-\frac{1}{2}}\right]\vm^{(k-1)}-\alpha_k(\widetilde\vV^{(k)})^{-\frac{1}{2}}\vg^{(k)} \\
=&~\frac{\beta_1}{1-\beta_1}\left[\vI - \alpha_k(\widetilde\vV^{(k)})^{-\frac{1}{2}}  \alpha_{k-1}^{-1}(\widetilde\vV^{(k-1)})^{\frac{1}{2}}  \right]\alpha_{k-1}(\widetilde{\vV}^{(k-1)})^{-\frac{1}{2}}\vm^{(k-1)}-\alpha_k(\widetilde\vV^{(k)})^{-\frac{1}{2}}\vg^{(k)}.
\end{align*}
The second equation of the above is exactly \eqref{zzandm}, and the last equation gives \eqref{zzandxx}.

\vspace{0.2cm}
\noindent\textbf{Proof of Lemma \ref{lem:terms1}}~~~~Inner producting $\nabla F(\vx^{(k)})$ with both sides of \eqref{zzandm} gives
\begin{equation}\label{eq:term1split}
\begin{aligned}
&~\nabla F(\vx^{(k)})^{\top}(\vz^{(k+1)}-\vz^{(k)})\\
= &~
\frac{\beta_1}{1-\beta_1}\nabla F(\vx^{(k)})^{\top}\left[\alpha_{k-1}(\widetilde\vV^{(k-1)})^{-\frac{1}{2}}-\alpha_k(\widetilde\vV^{(k)})^{-\frac{1}{2}}\right]\vm^{(k-1)}-\nabla F(\vx^{(k)})^{\top}\alpha_k(\widetilde\vV^{(k)})^{-\frac{1}{2}}\vg^{(k)}.
\end{aligned}
\end{equation}
We bound the first term on the right-hand-side of \eqref{eq:term1split} by Definition~\ref{def:fgbound} and Lemma~\ref{lem:indbnd} as follows:
\begin{align}
&~\nabla F(\vx^{(k)})^{\top}\left[\alpha_{k-1}(\widetilde\vV^{(k-1)})^{-\frac{1}{2}}-\alpha_k(\widetilde\vV^{(k)})^{-\frac{1}{2}}\right]\vm^{(k-1)}\nonumber\\
= &~\sum_{i=1}^{n}  \nabla_i F(\vx^{(k)}) \left[ \alpha_{k-1}(\widetilde v^{(k-1)}_i)^{-\frac{1}{2}}-\alpha_k(\widetilde v^{(k)}_i)^{-\frac{1}{2}} \right]  m^{(k-1)}_i \nonumber\\
\leq &~\sum_{i=1}^{n}  \Phi_i \left| \alpha_{k-1}(\widetilde v^{(k-1)}_i)^{-\frac{1}{2}}-\alpha_k(\widetilde v^{(k)}_i)^{-\frac{1}{2}} \right|\Gamma_i\cr
=  &~ \sum_{i=1}^{n}  \Gamma_i\Phi_i \left[\alpha_{k-1}(\widetilde v^{(k-1)}_i)^{-\frac{1}{2}}-\alpha_k(\widetilde v^{(k)}_i)^{-\frac{1}{2}} \right],\label{eq:term1-1}
\end{align}
where the last equation follows because  $\alpha_{k-1}(\widetilde\vv^{(k-1)})^{-\frac{1}{2}} \geq \alpha_k(\widetilde\vv^{(k)})^{-\frac{1}{2}}> \vzero$ component-wisely.
Similarly, we can bound the second term on the right-hand-side of \eqref{eq:term1split} as follows:
\begin{align}
&-\nabla F(\vx^{(k)})^{\top}\alpha_k(\widetilde\vV^{(k)})^{-\frac{1}{2}}\vg^{(k)}\nonumber\\
&=-\nabla F(\vx^{(k)})^{\top}\alpha_{k-1}(\widetilde\vV^{(k-1)})^{-\frac{1}{2}}\vg^{(k)}+ \nabla F(\vx^{(k)})^{\top}\left[\alpha_{k-1}(\widetilde\vV^{(k-1)})^{-\frac{1}{2}}-\alpha_k(\widetilde\vV^{(k)})^{-\frac{1}{2}}\right]\vg^{(k)}\nonumber\\
&= 
-\nabla F(\vx^{(k)})^{\top}\alpha_{k-1}(\widetilde\vV^{(k-1)})^{-\frac{1}{2}}\vg^{(k)}+
\sum_{i=1}^{n}\nabla_i F(\vx^{(k)}) \left[\alpha_{k-1}(\widetilde v^{(k-1)}_i)^{-\frac{1}{2}}-\alpha_k(\widetilde v^{(k)}_i)^{-\frac{1}{2}}\right]   g^{(k)}_i\nonumber\\
&\leq -\nabla F(\vx^{(k)})^{\top}\alpha_{k-1}(\widetilde\vV^{(k-1)})^{-\frac{1}{2}}\vg^{(k)}+ \sum_{i=1}^{n}  \Phi_i \left| \alpha_{k-1}(\widetilde v^{(k-1)}_i)^{-\frac{1}{2}}-\alpha_k(\widetilde v^{(k)}_i)^{-\frac{1}{2}} \right|\Gamma_i \nonumber\\
&=-\nabla F(\vx^{(k)})^{\top}\alpha_{k-1}(\widetilde\vV^{(k-1)})^{-\frac{1}{2}}\vg^{(k)}+\sum_{i=1}^{n}  \Gamma_i\Phi_i \left[\alpha_{k-1}(\widetilde v^{(k-1)}_i)^{-\frac{1}{2}}-\alpha_k(\widetilde v^{(k)}_i)^{-\frac{1}{2}} \right].
\label{eq:term1-2}
\end{align}
Now substituting \eqref{eq:term1-1} and \eqref{eq:term1-2} into \eqref{eq:term1split} yields \eqref{eq:terms1}.

\vspace{0.2cm}
\noindent\textbf{Proof of Lemma \ref{lem:terms23}}~~~~From \eqref{eq:zznormbound} and the fact $(a+b)^2\leq 4a^2+ \frac{4}{3}b^2,\,\forall\, a, b\in\RR$, the inequality in \eqref{eq:terms3} immediately follows.
By the Cauchy-Schwarz inequality, and also \eqref{eq:zxlip} and\eqref{eq:zznormbound}, it holds
\begin{align*}
&\left(\nabla F(\vz^{(k)})-\nabla F(\vx^{(k)})\right)^{\top} (\vz^{(k+1)}-\vz^{(k)})  \nonumber\\
\leq &~ \|\nabla F(\vz^{(k)})-\nabla F(\vx^{(k)}) \| \cdot \|\vz^{(k+1)}-\vz^{(k)}\|  \nonumber\\
{\leq} &~ \frac{L\beta_1 }{1-\beta_1} \|\vx^{(k-1)}-\vx^{(k)}\| \left(\frac{\beta_1}{1-\beta_1} \|\vx^{(k-1)}-\vx^{(k)}\|+\|\alpha_k(\widetilde\vV^{(k)})^{-\frac{1}{2}}\vg^{(k)}\|\right)   \nonumber\\
= &~\frac{L\beta_1^2}{(1-\beta_1)^2} \|\vx^{(k-1)}-\vx^{(k)}\|^2 + \frac{\beta_1 L}{1-\beta_1} \|\vx^{(k-1)}-\vx^{(k)}\| \cdot \|\alpha_k(\widetilde\vV^{(k)})^{-\frac{1}{2}}\vg^{(k)}\|.
\end{align*}
Now using the Young's inequality, we have \eqref{eq:terms2} from the above inequality.

\end{document}